\pdfoutput=1
\documentclass{shinyart}

\usepackage[utf8]{inputenc}

\usepackage{shinybib}

\usepackage{booktabs}
\usepackage{enumitem}
\usepackage{autonum}

\newcommand{\eps}{\varepsilon}
\renewcommand{\phi}{\varphi}
\newcommand{\N}{\mathbb{N}}
\newcommand{\R}{\mathbb{R}}
\newcommand{\Rbar}{\overline{\mathbb{R}}}
\newcommand{\calF}{\mathcal{F}}
\newcommand{\calG}{\mathcal{G}}

\newcommand{\calJ}{\mathcal{J}}

\newcommand{\calS}{\mathcal{S}}
\newcommand{\calL}{\mathcal{L}}
\newcommand{\scalprod}[1]{\langle #1 \rangle}
\newcommand{\norm}[1]{\| #1 \|}
\newcommand{\set}[2]{\left\{#1:#2\right\}}
\newcommand{\wkto}{\rightharpoonup}

\DeclareMathOperator{\dom}{\mathrm{dom}}

\DeclareMathOperator{\Id}{\mathrm{Id}}

\usepackage{tikz,pgfplots}
\pgfplotsset{compat=newest}
\pgfplotsset{plot coordinates/math parser=false}

\title{A convex analysis approach to multi-material topology optimization}
\author{Christian Clason\thanks{Faculty of Mathematics, University Duisburg-Essen, 45117 Essen, Germany (\email{christian.clason@uni-due.de})}
    \and Karl Kunisch\thanks{Institute of Mathematics and Scientific Computing, University of Graz, Heinrichstrasse 36, 8010 Graz, Austria, and Radon Institute, Austrian Academy of Sciences, Linz, Austria
    (\email{karl.kunisch@uni-graz.at}).}
}
\date{January 14, 2016}

\hypersetup{
    pdftitle={A convex analysis approach to multi-material topology optimization},
    pdfauthor={Christian Clason, Karl Kunisch},
    pdfkeywords={Topology optimization, convex analysis, convexification, semi-smooth Newton method}
}
\begin{document}

\maketitle
\allowdisplaybreaks

\begin{abstract}
    This work is concerned with optimal control of partial differential equations where the control enters the state equation as a coefficient and should take on values only from a given discrete set of values corresponding to available materials. A ``multi-bang'' framework based on convex analysis is proposed where the desired piecewise constant structure is incorporated using a convex penalty term. Together with a suitable tracking term, this allows formulating the problem of optimizing the topology of the distribution of material parameters as minimizing a convex functional subject to a (nonlinear) equality constraint. The applicability of this approach is validated for two model problems where the control enters as a potential and a diffusion coefficient, respectively. This is illustrated in both cases by numerical results based on a semi-smooth Newton method.
\end{abstract}

\section{Introduction}
\label{sec:introduction}

In this work, topology optimization
consists in determining the optimal distribution of two or more given materials within a domain, where the material properties enter as the values of a spatially varying coefficient $u(x)$ into the operator of a  partial differential equation.
We propose to follow a direct approach and minimize a cost functional of interest subject to the constraint $u(x)\in\{u_1,\dots,u_d\}$, where $u_i$ are given parameters specific to different materials.  This constraint is realized by means of the penalty functional
\begin{equation}
    \calG_0(u) = \int_\Omega \frac\alpha{2} |u(x)|^2 + \beta \prod_{i=1}^d |u(x)-u_i|^0 \,dx,
\end{equation}
where $|0|^0 = 0$ and $|t|^0 = 1$ for $t\neq 0$, and $\alpha$ and $\beta$ are fixed parameters to be further discussed below (see \cref{cor:subdiff}).
This functional was analyzed  in \cite{CK:2013} in the context of linear optimal control problems. There it was shown that, under mild technical assumptions,  the solutions to  optimal control problems based on the
convex envelope $\calG_\Gamma$ of $\calG_0$ have the desired property of being exactly multi-bang.  This means  that the solutions assume values in $\{u_1,\dots, u_d\}$ pointwise a.e. in the control domain, provided that $\beta$ is sufficiently large. This property is related to the use of the $\ell^1$ norm in sparse optimization as the convex envelope (on the unit interval) of the $\ell^0$ ``norm''. Although the explicit form of $\calG_\Gamma$ is not needed in our approach, we compute it in \cref{sec:l1} and remark on its relation to a direct $L^1$-type penalization of the constraint $u(x)\in\{u_1,\dots,u_d\}$.

In this work, we focus on  tracking-type functionals for multi-material optimization, i.e., we consider the optimization problem
\begin{equation}
    \label{eq:formal_prob}
    \min_{u\in U} \frac12 \norm{S(u)-z}_Y^2 + \calG_\Gamma(u),
\end{equation}
where
\begin{equation}
    U=\set{u\in L^2(\Omega)}{u(x)\in [u_1,u_d]\quad \text{for almost all } x\in\Omega}
\end{equation}
is the admissible set with $u_1<\dots<u_d$ given, $Y$ is a Hilbert space, $z\in Y$ is the given desired state, and $S:U\to Y$ is the (nonlinear) parameter-to-state mapping.

Following \cite{CK:2013,CIK:2014}, we can derive a first-order necessary primal-dual optimality system
\begin{equation}\label{eq:formal_opt}
    \left\{\begin{aligned}
            -\bar p &= S'(\bar u)^*(S(\bar u)-z),\\
            \bar u &\in \partial\calG_0^*(\bar p)
    \end{aligned}\right.
\end{equation}
(where $\partial\calG_0^*$ is the convex subdifferential of the (convex) Fenchel conjugate of $\calG_0$),
whose Moreau--Yosida regularization is amenable to numerical solution by a superlinearly convergent semismooth Newton method. While in earlier works, we considered the case of linear $S$, the main focus here is on
nonlinear, and in particular bilinear, parameter-to-state mappings. Our aim is to demonstrate that the proposed methodology provides a viable technology for solving multi-material shape and topology optimization problems without the need for computing shape or topological derivatives.

Let us very briefly point out some of the alternative approaches for topology optimization and give very selective references.
Relaxation methods \cite{Allaire:2002, Bendsoe2003, Pironneau1984, Sprekels2006} are amongst the earliest and most frequently used techniques.
A standard approach for the two-material case  consists in setting  $u(x)=u_1 w(x) + u_2(1-w(x))$ and minimizing over the set of all characteristic functions $w(x)\in\{0,1\}$. This problem is non-convex, but its convex relaxation -- minimizing over all $w(x)\in[0,1]$ -- often has a bang-bang solution, i.e., $w(x)\in\{0,1\}$ almost everywhere. For multi-material optimization, this approach can be extended by introducing multiple characteristic functions; non-overlapping materials can be enforced by considering the third domain as an intersection of two (possibly overlapping) domains, e.g., $u(x) = u_1w_1(x) + u_2 (1-w_1(x))w_2(x) + u_3 (1-w_1(x))(1-w_2(x))$ for $w_1(x),w_2(x)\in[0,1]$. For an increasing number $d$ of materials, this approach has obvious drawbacks due to the combinatorial nature and increasing non-linearity.
Shape calculus techniques \cite{Pironneau1984, Sokolowski1992} focus on the effect of smooth perturbations of the interfaces on the cost functional and have reached a high level of sophistication. From the point of view of numerical optimization, they are first-order methods and stable, with the drawback that they mostly allow only smooth variations  of the reference geometry. When combined with level-set techniques \cite{ Allaire2004, Ito2001}, they are flexible enough to allow vanishing and merging of connected components, but they do not allow the creation of holes. This is allowed in the context of  topological sensitivity analysis \cite{Garreau2001, Sokolowski1999}, which investigates the effect of the creation of holes on the cost. Let us point out that in our work we do not rely in any explicit manner on knowledge of the shape or the topological derivatives. Moreover, the numerical technique that we propose is of second order rather than of gradient nature. Second-order shape  or topological derivative analysis is available, but it is involved when it comes to numerical realization. Multi-material optimization for elasticity problems are further investigated in \cite{Haslinger:2010} by means of H-convergence methods and by phase-field methods in \cite{Blank:2014}.
The work which in part is most closely related to ours is \cite{Amstutz2011}, see also \cite{Amstutz2006, Amstutz2010}, where for the case of linear solution operators and two materials, the set of coefficients is expressed in terms of characteristic functions, and the resulting problem is considered in function spaces rather than in terms of subdomains and their boundaries. The first order-optimality condition is derived and formulated as a nonlinear equation for which a semi-smooth Newton method is applicable.

The general theory to be developed will be tested on two particular model problems. For the first one, the mapping $S:u\mapsto y\in H^2(\Omega)$ is the solution operator to
\begin{equation}\label{eq:prob_pot}
    \left\{\begin{aligned}
            -\Delta y + uy &= f,\\
            \partial_\nu y &= 0,
    \end{aligned}\right.
\end{equation}
for $u$ in an appropriate subset of $L^2(\Omega)$ and fixed $f\in L^2(\Omega)$.
The second one is motivated by the  mapping
$\tilde S:u \mapsto y\in H^1_0(\Omega)$, where $y$ is the solution to
\begin{equation}\label{eq:prob_diff}
    \left\{\begin{aligned}
            -\nabla\cdot(u\nabla y) &= f,\\
            y &= 0,
    \end{aligned}\right.
\end{equation}
with $u$ in a subset of $L^\infty(\Omega)$. It is well known from \cite{Murat:1977} that \eqref{eq:formal_prob} does not admit a solution in this case, since the differential equation is not closed under  weak-$*$  convergence in $L^\infty(\Omega)$.
For this reason we  shall introduce a local smoothing operator $G$ and define the associated solution operator as $S=\tilde S\circ G$. We point out that the operator to be used in \cref{sec:solution} will be of local nature. It acts as smoothing of the constant values $u_i$ across interior interfaces of boundaries between different materials and will justify the use of a semi-smooth Newton method for the numerical realization.

This work is organized as follows. In \cref{sec:existence}, existence of a solution to \eqref{eq:formal_prob} is shown and the explicit form of \eqref{eq:formal_opt} is derived. \Cref{sec:l1} is devoted to the explicit form of $\calG$ and its comparison to an alternative $L^1$-type penalty. The numerical solution is addressed in \cref{sec:solution}, where the Moreau--Yosida regularization and its convergence are treated for general nonlinear mappings in \cref{sec:solution:regularization}. The analysis of the semismooth Newton method for the regularized problems requires specific properties of the state equation and is therefore addressed in \cref{sec:solution:ssn} separately for each model problem. Finally, numerical results are presented in \cref{sec:examples}.

\section{Existence and optimality conditions}
\label{sec:existence}

We set
\begin{align}
    &\calF:L^2(\Omega)\to \Rbar, \qquad &\calF(u) &= \frac12\norm{S(u)-z}^2_Y,\\
    &\calG_0:L^2(\Omega)\to \Rbar,\qquad &\calG_0(u) &= \frac\alpha2\norm{u}^2_{L^2} + \beta\int_\Omega \prod_{i=1}^d |u(x)-u_i|^0 \,dx + \delta_U(u),
\end{align}
where $U\subset L^2(\Omega)$ is a convex and closed set and $\delta_U$ is the indicator function in the sense of convex analysis, i.e.,
\begin{equation}\label{eq:indicator}
\delta_U(u) = \begin{cases} 0 & \text{if } u\in U,\\ \infty & \text{if } u\notin U.\end{cases}
\end{equation}
For $S:U\to Y$, we assume that
\begin{enumerate}[label= (\textsc{a}\arabic{enumi}), ref=\textsc{a}\arabic{enumi},align=left]
    \item $S:U\to Y$ is weak-to-weak continuous, i.e., $\{u_n\}_{n\in\N}\subset U$ and $u_n\wkto u\in U$ in $L^2(\Omega)$ implies $S(u_n)\wkto S(u)\in Y$;\label{ass:a1}
    \item $S$ is twice Fréchet differentiable.\label{ass:a2}
\end{enumerate}
Both assumptions are satisfied for the two model problems stated in the introduction.
Now consider
\begin{equation}
    \label{eq:problem}
    \min_{u\in L^2(\Omega)} \calF(u) + \calG(u)
\end{equation}
for
\begin{equation}
    \calG := \calG_0^{**},
\end{equation}
where $\calG_0^{**}$ is the biconjugate of $\calG_0$, i.e., the Fenchel conjugate of
\begin{equation}
    \calG^*_0: L^2(\Omega)\to \Rbar,\qquad \calG^*_0(q) = \sup_{u\in L^2(\Omega)} \scalprod{q,u} - \calG_0(u).
\end{equation}
Since Fenchel conjugates are always lower semicontinuous and convex, see, e.g. \cite[Proposition 13.11]{Bauschke:2011}, it follows that $\calG$ is proper,  lower semicontinuous and convex for any  $\alpha >0$ and $\beta\geq 0$.
Existence of a solution to \eqref{eq:formal_prob} thus follows under the stated assumptions on $S$.
\begin{proposition}\label{thm:existence}
    There exists a solution $\bar u\in U$ to \eqref{eq:formal_prob} for any $\alpha >0$ and $\beta\geq 0$.
\end{proposition}
\begin{proof}
    Due to Assumption \eqref{ass:a1}, the tracking term $\calF$ is weakly lower semicontinuous and bounded from below. Similarly, $\calG_0$ is bounded from below by $0$, which implies that $\calG_0^{**}\geq 0$ as well, see, e.g. \cite[Proposition 13.14]{Bauschke:2011}. Since $U$ is a compact subset of $L^2(\Omega)$, we have
    \begin{equation}
        U=\dom \calG_0 \subset \dom \calG_0^{**} \subset \overline{\dom}\,\calG_0 = \overline{U}=U,
    \end{equation}
    see, e.g., \cite[Proposition 13.40]{Bauschke:2011},
    and hence that $\calG=\calG_0^{**}$ is coercive. This implies that $\calF+\calG$ is proper, weakly lower semicontinous and coercive, and application of Tonelli's direct method yields existence of a minimizer.
\end{proof}

We next derive first-order necessary optimality conditions of primal-dual type.
\begin{proposition}\label{thm:optsys}
    Let $\bar u\in U$ be a local minimizer of \eqref{eq:problem}. Then there exists a $\bar p \in L^2(\Omega)$ satisfying
    \begin{equation}
        \label{eq:optsys}
        \left\{\begin{aligned}
                -\bar p &= S'(\bar u)^*(S(\bar u)-z),\\
                \bar u &\in \partial\calG^*(\bar p).
        \end{aligned}\right.
    \end{equation}
\end{proposition}
\begin{proof}
    Let $\bar u\in U$ be a local minimizer, i.e., for $t>0$ small enough and any $u\in U$ there holds
    \begin{equation}
        \label{eq:optsys1}
        \calF(\bar u) + \calG(\bar u) \leq \calF(\bar u+t(u-\bar u)) + \calG(\bar u+t(u-\bar u)).
    \end{equation}
    Since $\calG$ is convex, we have
    \begin{equation}
        \calG(\bar u + t(u-\bar u)) =  \calG(t u + (1-t)\bar u) \leq t\calG(u) + (1-t)\calG(\bar u),
    \end{equation}
    which implies
    \begin{equation}
        \label{eq:optsys2}
        \calG(t u + (1-t)\bar u) - \calG(\bar u) \leq t (\calG(u)-\calG(\bar u)).
    \end{equation}
    Inserting this in \eqref{eq:optsys1} and rearranging yields
    \begin{equation}
        \calF(\bar u+t(u-\bar u)) - \calF(\bar u) + t (\calG(u)-\calG(\bar u)) \geq 0.
    \end{equation}
    Since $\calF$ is Fréchet-differentiable due to Assumption \eqref{ass:a2}, we can divide by $t>0$ and let $t\to 0$ to obtain
    \begin{equation}
        \scalprod{\calF'(\bar u) ,u-\bar u}  + \calG(u)-\calG(\bar u) \geq 0
    \end{equation}
    for every $u\in U$, i.e.,
    \begin{equation}
        \bar p:= -\calF'(\bar u) \in \partial\calG(\bar u).
    \end{equation}
    Since $\calG$ is convex, this is equivalent to $\bar u\in \partial\calG^*(\bar p)$. Applying the chain rule for Fréchet derivatives to $\calF$ then yields the desired optimality conditions.
\end{proof}
The question of optimality of solutions to Problem \eqref{eq:problem} with respect to the non-convex functional $\calF+\calG_0$ has been addressed (for linear $S$) in \cite{CK:2013}; here we only remark that since $\calG=\calG_0^{**}\leq \calG_0$ and $\calG(u) = \calG_0(u)$ for $u(x)\in\{u_1,\dots,u_d\}$ almost everywhere (see \cref{sec:l1} below), it follows that if a (local) minimizer $\bar u$ of \eqref{eq:problem} satisfies $\bar u(x)\in\{u_1,\dots,u_d\}$ almost everywhere, we have for all $u\in U$ (sufficiently close to $\bar u$) that
\begin{equation}
    \calF(u) + \calG_0(u) \geq \calF(u) + \calG(u) \geq  \calF(\bar u) + \calG(\bar u) =  \calF(\bar u) + \calG_0(\bar u),
\end{equation}
i.e., $\bar u$ is a (local) minimizer of $\calF +\calG_0$ as well.

Since $\calG^* = (\calG_0^{**})^{*} = \calG_0^{***} = \calG_0^*$, see, e.g., \cite[Proposition 13.14\,(iii)]{Bauschke:2011}, we can make use of the following characterization from \cite[\S\,2.1]{CK:2013}.
\begin{corollary}\label{cor:subdiff}
    If $\alpha$ and $\beta$ satisfy the relation
    \begin{equation}
        \label{eq:cond_beta}
        \frac\alpha2 (u_{i+1}-u_i) \leq \sqrt{2\alpha\beta} \quad\text{for all } 1\leq i<d,
    \end{equation}
    then $u\in \partial\calG^*(p)$ if and only if for almost all $x\in \Omega$,
    \begin{equation}\label{eq:kk1}
        u(x) \in
        \begin{cases}
            \{u_1\} & p(x)<\frac\alpha2 (u_1+u_2),\\
            \{u_i\} & \frac\alpha2(u_{i-1}+u_i) < p(x) < \frac\alpha2(u_{i}+u_{i+1}),\qquad  1<i<d,\\
            \{u_d\} & p(x)>\frac\alpha2 (u_{d-1}+u_d),\\
            [u_{i},u_{i+1}] & p(x) = \frac\alpha2(u_{i}+u_{i+1}), \qquad 1\leq i < d.
        \end{cases}
    \end{equation}
\end{corollary}

Thus, with \eqref{eq:cond_beta} holding, $u(x)$ coincides with one of the preassigned control values $u_i$, except in the singular cases when $p(x)=\frac\alpha2(u_{i}+u_{i+1})$ for some $i$. If, on the other hand, \eqref{eq:cond_beta} is not satisfied, then $u=\frac{1}{\alpha}p$ may hold on subsets $\hat \Omega$ of nontrivial measure. In this case we call $u|_{\hat \Omega}$ a \emph{free arc}, and refer to \cite{CK:2013} for details.

\section{Relation to \texorpdfstring{$\scriptstyle L^1$}{L¹} penalization}
\label{sec:l1}

We now compare the penalty $\calG$ to a direct $L^1$ penalization of $u(x)-u_i$, $i\in\{1,\dots,d\}$. First, we give an explicit characterization of $\calG=\calG_0^{**}$. Since $\calG_0$ is defined via the integral of a pointwise function of $u(x)$,
we can compute the Fenchel conjugate and its subdifferential pointwise as well; see, e.g., \cite[Props.~IV.1.2, IX.2.1]{Ekeland:1999a}, \cite[Prop.~16.50]{Bauschke:2011}.
It therefore suffices to consider
\begin{equation}
    g_0:\R \to \Rbar,\qquad g_0(v) = \frac\alpha2 |v|^2 + \beta \prod_{i=1}^d |v-u_i|^0 + \delta_{[u_1,u_d]}(v),
\end{equation}
where $\delta_{[u_1,u_d]}$ is again the indicator function in the sense of convex analysis, cf.~\eqref{eq:indicator}.
To compute $g_0^{**}$ we make use of the fact that the biconjugate coincides with the lower convex envelope (or Gamma-regularization)
\begin{equation}
    g_\Gamma(v) = \sup\set{a(v)}{a:\R\to\R\text{ is affine and }a\leq g_0},
\end{equation}
see, e.g., \cite[Theorem 2.2.4\,(a)]{Schirotzek:2007}. We assume again that \eqref{eq:cond_beta} holds.

First, note that $g_0(u_i) = \frac\alpha2u_i^2$ for all $1\leq i\leq d$, which implies that $g_\Gamma(u_i)\leq \frac\alpha2u_i^2$. Now consider a single interval $[u_i,u_{i+1}]$ for $1\leq i<d$. Obviously, a candidate for $g_\Gamma(v)$ in $v\in\{u_i,u_{i+1}\}$ is given by the linear interpolant $g_i$ of $g_0(u_i)$ and $g_0(u_{i+1})$, i.e.,
\begin{equation}
    g_i(v) = \frac\alpha 2\left((u_i+u_{i+1})v - u_i u_{i+1}\right).
\end{equation}
This function in fact satisfies the conditions for $g_\Gamma$ also for $v\in (u_i,u_{i+1})$, which follows from the fact that on this open interval, the quadratic function
\begin{equation}
    (g_0-g_i)(v) = \frac\alpha 2\left(v^2-(u_i+u_{i+1})v + u_i u_{i+1}\right) + \beta
\end{equation}
has a unique minimizer (since $\alpha>0$) in its critical point $\bar v = \frac12 (u_i+u_{i+1})$, where
\begin{equation}
    \begin{aligned}
        (g_0-g_i)(v) &= \frac\alpha2 \left(-\frac14(u_i+u_{i+1})^2 + u_iu_{i+1}\right)+\beta\\
                     &= -\frac\alpha 8\left(u_{i+1}-u_i\right)^2  + \beta\geq 0
    \end{aligned}
\end{equation}
by \eqref{eq:cond_beta}. Hence, $g_i(v) \leq g_0(v)$ for all $v\in[u_i,u_{i+1}]$ with equality in $v\in\{u_i,u_{i+1}\}$.

To obtain a global function, we define $\bar g:[u_1,u_d]\to \R$ via
\begin{equation}
    \bar g(v) := g_i(v) \qquad\text{for } v\in[u_i,u_{i+1}],\quad 1\leq i<d.
\end{equation}
It remains to verify that for each fixed $i$, we have $g_j(v)\leq g_i(v)$ for all $j\neq i$ and $v\in[u_i,u_{i+1}]$. A short computation shows that
$g_j(u_{i}) \leq g_{i}(u_{i})$. Moreover, due to the ordering of the $u_i$ we have
\begin{equation}
    g_j'(v) = \frac\alpha2(u_j+u_{j+1}) > \frac\alpha2(u_{i+1}+u_{i+2}) = g_{i}'(v)
\end{equation}
for all $j>i$ and similarly $g_i'(v)<g_j'(v)$ for all $j<i$. This implies that $g_j(v)\leq g_i(v)$ for all $j\neq i$ and $v\in[u_i,u_{i+1}]$.
Using again that $\dom g_\Gamma = \dom g_0 = [u_1,u_d]$ since the interval is closed, we obtain
\begin{equation}
    \label{eq:g_gamma}
    \begin{aligned}
        g_0^{**}(v) &= g_\Gamma(v) = \bar g(v)+\delta_{[u_1,u_d]}(v) \\
                    &= 
        \begin{cases}
            \frac\alpha 2\left((u_i+u_{i+1})v - u_i u_{i+1}\right) & v\in[u_i,u_{i+1}],\quad 1\leq i <d,\\
            \infty & v\in \R\setminus[u_1,u_d].
        \end{cases}
    \end{aligned}
\end{equation}
and hence
\begin{equation}
    \calG(u) = \int_\Omega g_\Gamma(u(x))\,dx.
\end{equation}

From the above, we have that $g_\Gamma$ is the unique continuous and piecewise (on $[u_i,u_{i+1}]$) affine function with $g_\Gamma(u_i) = \frac\alpha2 u_i^2$. It is not surprising that using such a function in optimization promotes solutions lying in the ``kinks'' (cf. sparse optimization using $\ell_1$-type norms, where the only ``kink'' is at $v=0$). Other penalties $h$ with a similar piecewise affine structure can be constructed by prescribing different values for $h(u_i)$, although the obvious choice $h(u_i)=\alpha|u_i|$ results in a shifted $\ell_1$ norm which has only one ``kink'' at $v=\min_i |u_i|$ and hence does not have the desired structure.

An alternative to this piecewise affine construction is the direct $\ell^1$-penalization of the deviation, i.e., choosing
\begin{equation}
    h(v) = \alpha \sum_{i=1}^d |v-u_i| + \delta_{[u_1,u_d]}(v).
\end{equation}
(Note that the product $\prod_{i=1}^d |v-u_i|$ is a polynomial of order $d$ and hence in general is not convex.) We first point out that the value $h(u_i)$ depends on all $u_j$, $1\leq j\leq d$, (and in particular, on $d$) rather than on $u_i$ only, which may be undesirable; see \cref{fig:gbicjonj}.
\begin{figure}
    \centering
\def\numsamples{301}
\begin{tikzpicture}[baseline,style={font=\small}]
\begin{axis}[%
width=0.5\textwidth,
xmin=-1, xmax=2,
ymin=0.2, ymax=1.26,
legend style={draw=none},
legend pos=north west,
xlabel style={at={(axis cs:2.1,0.25)}},
xlabel={$v$},
xtick=\empty,
extra x tick style={grid=major},
extra x ticks={-1,1,2},
extra x tick labels={$u_1$,$u_2$,$u_3$},
axis y line=left,
axis x line=bottom,
]
\addplot[
domain=-1:2,
         color=DarkBlue,solid,line width=1.5pt] 
{
    and(x>=-1,x<1)*0.25*((-1+1)*x -(-1)*(1)) +
    and(x>=1,x<=2)*0.25*((1+2)*x -(1)*(2))
};
\addlegendentry{$g_0^{**}$};
\addplot[
domain=-1:2, color=DarkGreen,dashed,line width=1.5pt]
{
    0.25*pow(x,2)+0.26
};
\addlegendentry{$g_0$};
\addplot[samples at = {-1,1,2},
domain=-1:2, color=DarkGreen,mark=*, mark size=2.5pt,only marks]
{
(0.25*pow(x,2))
};
\addplot[samples at = {-1,1,2},
domain=-1:2, color=DarkGreen,mark=*,fill=white, mark size=2.5pt,only marks]
{
(0.25*pow(x,2))+0.26
};
\end{axis}
\end{tikzpicture}
\def\numsamples{201}
\begin{tikzpicture}[baseline,style={font=\small}]
\begin{axis}[%
width=0.5\textwidth,
xmin=-1, xmax=2,
ymin=1.5, ymax=2.5,
xlabel style={at={(axis cs:2.1,1.55)}},
xlabel={$v$},
xtick=\empty,
extra x tick style={grid=major},
extra x ticks={-1,1,2},
extra x tick labels={$u_1$,$u_2$,$u_3$},
axis y line=left,
axis x line=bottom,
]
\addplot[
            domain=-1:2,
            color=DarkBlue,solid,line width=1.5pt] 
{
    0.5*(abs(x-(-1))+abs(x-(1))+abs(x-(2)))
};
\end{axis}
\end{tikzpicture}
    \caption{Plot of $g_0^{**}$ and $g_0$ (left), $h$ (right) for $d=3$, $(u_1,u_2,u_3)$, $\alpha=0.5$, $\beta=0.26$ (satisfying \eqref{eq:cond_beta})}
    \label{fig:gbicjonj}
\end{figure}
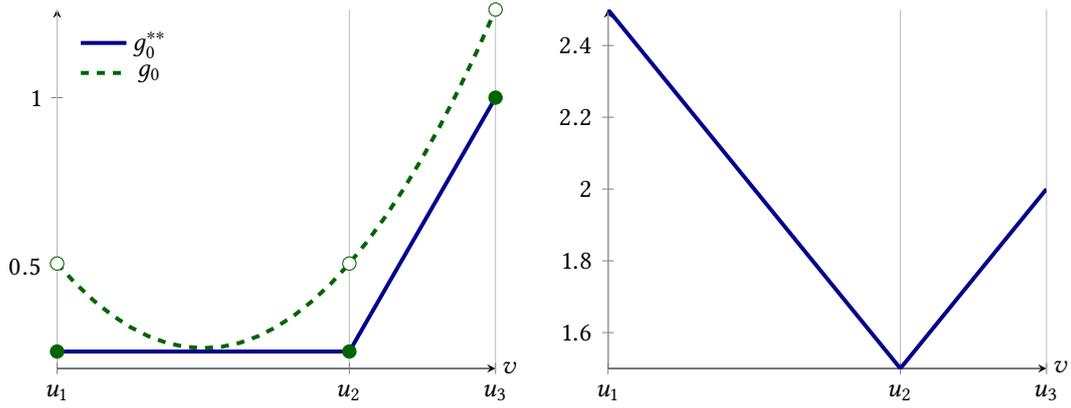
To further illustrate the practical difference between using $g_\Gamma$ and $h$, we compute the corresponding subdifferential $\partial h^*$ which would appear in \eqref{eq:optsys}. First, we determine the Fenchel conjugate
\begin{equation}
    \label{eq:l1:sup}
    h^*(q) = \sup_{v\in[u_1,u_d]} vq - \alpha\sum_{i=1}^d |v-u_i|.
\end{equation}
Since the function to be maximized is continuous and piecewise affine on $\R$, the supremum must be attained at $\bar v = u_i$ for some $1\leq i\leq d$. Making use of the fact that the $u_i$ are ordered, we obtain that $h^*(q)$ must be equal to one of the functions
\begin{equation}
    \begin{aligned}
        h_i^*(q) &= qu_i - \alpha \left(\sum_{j=1}^{i-1}(u_i-u_j) +\sum_{j=i+1}^{d}(u_j-u_i)\right) \\
                 &=u_i (q+\alpha(d+1-2i))+\alpha\sum_{j=1}^{i-1}u_j - \alpha\sum_{j=i+1}^{d}u_j
    \end{aligned}
\end{equation}
(with the convention that empty sums evaluate to $0$). It remains to determine the supremum over $1\leq i\leq d$ based on the value of $q$. For this, we first compare $h_i^*(q)$ with $h_{i+1}^*(q)$. Simple rearrangement of terms shows that $h_i^*(q)\leq h_{i+1}^*(q)$ if and only if
\begin{equation}
    \alpha(2i-d)(u_{i+1}-u_i)\leq q(u_{i+1}-u_i).
\end{equation}
Since $u_{i+1}>u_i$, we deduce that this is the case if and only if $q\geq \alpha(2i-d)$. Hence, the supremum is attained for the largest $i$ for which $q\geq \alpha(2i-d)$. This yields
\begin{equation}
    h^*(q) = \begin{cases}
        u_1 (q+\alpha(d-1)) - \alpha\sum_{j=2}^d u_j & \frac1\alpha q<2-d,\\
        u_i(q+\alpha(d+1-2i)) - \alpha\sum_{j=1}^{i-1}u_j + \alpha\sum_{j=i+1}^{d}u_j & 2(i-1)-d \leq \frac1\alpha q < 2i-d, \ 1< i<d,\\
        u_d(q-\alpha(d+1)) + \alpha\sum_{j=1}^{d-1}u_j & \frac1\alpha q\geq d-2.
    \end{cases}
\end{equation}
Since $h^*$ is continuous and piecewise differentiable, we have that the convex subdifferential is given by
\begin{equation}
    \partial h^*(q) =
    \begin{cases}
        \{u_1\} & \frac1\alpha q<2-d,\\
        \{u_i\}& 2(i-1)-d < \frac1\alpha q < 2i-d, \quad 1< i<d,\\
        \{u_d\} & \frac1\alpha q>d-2,\\
        [u_i,u_{i+1}] & \frac1\alpha q = 2i-d, \quad 1\leq i < d.
    \end{cases}
\end{equation}
Comparing this with \cref{cor:subdiff}, we see that the case distinction is independent of $u_i$, but rather depends on $d$ only, with the individual cases always being intervals of length $2\alpha$. In particular, for fixed $q$, the value $\partial h^*(q)$ changes if the number of parameters $d$ is increased, independent of the magnitude of the additional parameters. Furthermore, since the distribution of intervals is symmetric around the origin, $h$ tends to favor for increasing $\alpha$ those $u_i$ closer to the ``middle parameter'' $u_{d/2}$, rather than those of smaller magnitude as is the case for $g_0^{**}$; see \cref{fig:hsubdiff}.
\begin{figure}
    \centering
\def\numsamples{1001}
\begin{tikzpicture}[baseline,style={font=\small}]
\begin{axis}[%
width=0.5\textwidth,
xmin=-1, xmax=1.5,
ymin=-1.2, ymax=2,
xtick=\empty,
extra x tick style={grid=major},
extra x ticks={0,0.75},
extra x tick labels={$\frac{\alpha}{2} (u_1 +u_2)$, $\frac{\alpha}{2} (u_2 +u_3)$},
xlabel style={at={(axis cs:1.6,-1.2)}},
xlabel={$q$},
axis y line=left,
axis x line=bottom,
legend pos=north west,
]
\addplot[domain=-1:1.5,
samples=\numsamples,
color=DarkBlue,line width=1.5pt,solid] 
{
 (x<=0.25*(-1+1)) * (-1) +
 and(x>=0.25*(-1+1), x<0.25*(1+2)) * (1) +
 (x>=0.25*(1+2)) * (2)
};
\end{axis}
\end{tikzpicture}%
\def\numsamples{1001}
\begin{tikzpicture}[baseline,style={font=\small}]
\begin{axis}[%
width=0.5\textwidth,
xmin=-1, xmax=1.5,
ymin=-1.2, ymax=2,
xtick=\empty,
extra x tick style={grid=major},
extra x ticks={-0.5,0.5},
extra x tick labels={$\alpha(2-d)$,$\alpha(4-d)$},
xlabel style={at={(axis cs:1.6,-1.2)}},
xlabel={$q$},
axis y line=left,
axis x line=bottom,
legend pos=north west,
]
\addplot[domain=-1:1.5,
samples=\numsamples,
color=DarkBlue,line width=1.5pt,solid] 
{
    (x<=0.5*(2-3)) * (-1) +
    and(x>0.5*(2-3),x<=0.5*(2*2-3)) * (1) + 
    (x>0.5*(2*2-3)) * (2)
};
\end{axis}
\end{tikzpicture}%
    \caption{Plot of $\partial g^*$ (left), $\partial h^*$ (right) for $d=3$, $(u_1,u_2,u_3)$, $\alpha=0.5$, $\beta=0.26$}
    \label{fig:hsubdiff}
\end{figure}
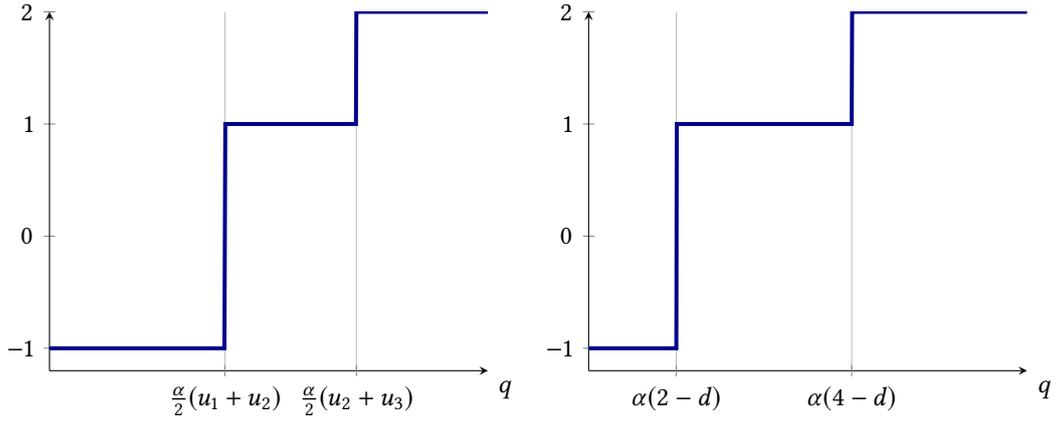

\section{Numerical solution}
\label{sec:solution}

For the numerical solution, we follow the approach described in \cite{CIK:2014} for linear parameter-to-state mappings, where we replace $\partial\calG^*$ by its Moreau--Yosida regularization and apply a semi-smooth Newton method with backtracking line search and continuation. In this section, we describe the necessary modifications for nonlinear mappings, arguing in terms of the functional instead of the optimality system. We first introduce the regularization and discuss its convergence to the original problem for general nonlinear mappings in \cref{sec:solution:regularization}. The explicit form and well-posedness of the Newton step (from which superlinear convergence follows) requires exploiting the structure of the mapping, hence we discuss it separately for each model problem in \cref{sec:solution:ssn}.

\subsection{Regularization}
\label{sec:solution:regularization}

Since $\calF$ is not convex, we cannot proceed directly to the regularized system. Instead, we start by considering for $\gamma>0$ the regularized problem
\begin{equation}
    \label{eq:problem_reg}
    \min_{u\in L^2(\Omega)} \calF(u) + \calG(u) + \frac\gamma2 \norm{u}_{L^2(\Omega)}^2.
\end{equation}
By the same arguments as in the proof of \cref{thm:existence}, we obtain the existence of a minimizer $u_\gamma\in U$. We now address convergence of $u_\gamma$ as $\gamma \to 0$.
\begin{proposition}
    The family $\{u_{\gamma}\}_{\gamma>0}$ of global minimizers to \eqref{eq:problem_reg} contains at least one subsequence $\{u_{\gamma_n}\}_{n\in\N}$ converging to a global minimizer of \eqref{eq:problem} as $n\to\infty$. Furthermore, for any such subsequence the convergence is strong.
\end{proposition}
\begin{proof}
    Since $U$ is bounded, the set $\{u_{\gamma}\}_{\gamma>0}$ contains a subsequence   $\{u_{\gamma_n}\}_{n\in\N}$ with $\gamma_n\to 0$ converging weakly to some $\bar u$. Furthermore, it follows that $\lim_{n\to\infty} \frac{\gamma_n}2\norm{u_{\gamma_n}}^2_{L^2(\Omega)} =0$.
    By the weak lower semicontinuity of $\calJ:=\calF+\calG$ and the optimality of $u_{\gamma_n}$, we thus have for any $u\in U$ that
    \begin{equation}
        \begin{aligned}
            \calJ(\bar u)  \leq \lim\inf_{n\to\infty}  \calJ(u_{\gamma_n})   &= \lim\inf_{n\to\infty}  \calJ(u_{\gamma_n}) +\frac{\gamma_n}2 \norm{u_{\gamma_n}}^2_{L^2(\Omega)}\\
                                                                             &\leq   \calJ(u) +\lim_{n\to\infty}  \frac{\gamma_n}2 \norm{u}^2_{L^2(\Omega)} =  \calJ(u),
        \end{aligned}
    \end{equation}
    i.e., $\bar u$ is a global minimizer of \eqref{eq:problem}.

    To show strong convergence, it suffices to show $\lim\sup_{n\to\infty} \norm{u_{\gamma_n}} \leq \norm{\bar u}$. This follows from
    \begin{equation}
        \calJ(u_{\gamma_n}) + \frac{\gamma_n}2 \norm{u_{\gamma_n}}^2_{L^2(\Omega)} \leq \calJ(\bar u) +  \frac{\gamma_n}2 \norm{\bar u}^2_{L^2(\Omega)} \leq \calJ(u_{\gamma_n}) +  \frac{\gamma_n}2 \norm{\bar u}^2_{L^2(\Omega)}
    \end{equation}
    for every $n\in\N$ due to the optimality of $u_\gamma$ and $\bar u$. Hence, $\norm{u_{\gamma_n}}_{L^2(\Omega)}\to\norm{\bar u}_{L^2(\Omega)}$, which together with weak convergence implies strong convergence in the Hilbert space $L^2(\Omega)$ of the subsequence.
\end{proof}

Arguing as in the proof of \cref{thm:optsys}, we obtain the abstract first-order necessary optimality conditions
\begin{equation}
    \left\{\begin{aligned}
            -p_\gamma &= \calF'(u_\gamma),\\
            u_\gamma &\in \partial(\calG_\gamma)^*(p_\gamma),
    \end{aligned}\right.
\end{equation}
where
\begin{equation}
    \calG_\gamma(u) := \calG(u) + \frac\gamma2 \norm{u}_{L^2(\Omega)}^2.
\end{equation}
We now use that $(\calG+\frac\gamma2 \norm{\cdot}_{L^2(\Omega)}^2)^*$ is equal to the infimal convolution of $\calG^*$ and $\frac1{2\gamma} \norm{\cdot}_{L^2(\Omega)}^2$, which in turn coincides with the Moreau envelope of $\calG^*$; see, e.g., \cite[Proposition 13.21]{Bauschke:2011}. Furthermore, the Moreau envelope is Fréchet-differentiable with Lipschitz-continuous gradient which coincides with the Moreau--Yosida regularization $(\partial \calG^*)_\gamma$ of $\partial \calG^*$; see, e.g., \cite[Proposition 12.29]{Bauschke:2011}.
We can therefore make use of the pointwise characterization of $H_\gamma:=(\partial\calG^*)_\gamma=\partial(\calG_\gamma)^*$ from \cite[Appendix~A.2]{CIK:2014},  assuming again that \eqref{eq:cond_beta} holds, to obtain
\begin{equation}\label{eq:hgamma}
    [H_\gamma(p)](x) =
    \begin{cases}
        u_i & p(x)\in Q^\gamma_i,\qquad 1\leq i\leq d,\\
        \tfrac1\gamma\left(p(x)-\tfrac\alpha2(u_i+u_{i+1})\right) & p(x) \in Q^\gamma_{i,i+1},\quad 1\leq i < d.
    \end{cases}
\end{equation}
where
\begin{align}
    Q_1^\gamma &= \set{q}{q< \tfrac\alpha2\left(\left(1+ \tfrac{2\gamma}{\alpha}\right)u_1+u_2\right)},\\
    Q_i^\gamma &= \set{q}{\tfrac\alpha2\left(u_{i-1} + \left(1 + \tfrac{2\gamma}{\alpha}\right)u_i\right) < q < \tfrac\alpha2\left(\left(1+ \tfrac{2\gamma}{\alpha}\right)u_i+u_{i+1}\right)}
    \quad \text{ for } 1<i<d,\\
    Q_d^\gamma &= \set{q}{\tfrac\alpha2\left(u_{d-1} + \left(1 + \tfrac{2\gamma}{\alpha}\right)u_d\right) < q},\\
    Q_{i,i+1}^\gamma &= \set{q}{\tfrac\alpha2\left(\left(1 + \tfrac{2\gamma}{\alpha}\right)u_i+u_{i+1}\right) \leq q \leq \tfrac\alpha2\left(u_i+\left(1+ \tfrac{2\gamma}{\alpha}\right)u_{i+1}\right)} \quad\text{for } 1 \leq i < d,
\end{align}
to obtain the explicit primal-dual first-order necessary conditions
\begin{equation}\label{eq:opt_reg}
    \left\{\begin{aligned}
            -p_\gamma &= S'( u_\gamma)^*(S(u_\gamma)-z),\\
            u_\gamma &= H_\gamma(p_\gamma).
    \end{aligned}\right.
\end{equation}
Comparing \eqref{eq:hgamma} to \eqref{eq:kk1}, we observe that the Moreau--Yosida regularization is of local nature, acting along interfaces between regions with different material parameters.

Since $H_\gamma$ is a superposition operator defined by a Lipschitz continuous and piecewise differentiable scalar function, $H_\gamma$ is Newton-differentiable from $L^r(\Omega)\to L^2(\Omega)$ for any $r>2$; see, e.g., \cite[Example 8.12]{Kunisch:2008a} or \cite[Theorem 3.49]{Ulbrich:2011}. Its Newton derivative at $p$ in direction $h$ is given pointwise almost everywhere by
\begin{equation}
    [D_N H_\gamma(p)h](x) =
    \begin{cases}
        \frac1\gamma h(x) & \text{if }p(x)\in Q^\gamma_{i,i+1},\quad 1\leq i < d,\\
        0 &\text{else.}
    \end{cases}
\end{equation}

\subsection{Semismooth Newton method}
\label{sec:solution:ssn}

We now wish to apply a semismooth Newton method to \eqref{eq:opt_reg}. For this purpose, we need to argue that $p_\gamma\in V$ for some $V\hookrightarrow L^r(\Omega)$ with $r>2$ and show uniform invertibility of the Newton step. Since the control-to-state mapping is nonlinear, this requires exploiting its concrete structure. We thus directly consider the specific model problems.

\subsubsection{Potential problem}

We first express \eqref{eq:opt_reg} in equivalent form by introducing the state $y_\gamma=S(u_\gamma)\in H^1(\Omega)$, i.e., satisfying for $u=u_\gamma$
\begin{equation}\label{eq:pot_state}
    \left\{\begin{aligned}
            -\Delta y + u y &= f  &\text{ in } \Omega,\\
            \partial_\nu y &= 0&\text{ on } \partial \Omega.
    \end{aligned}\right.
\end{equation}
In the following, we assume that $\Omega\subset \R^N$, $N \le 3$, is sufficiently regular such that for any $f\in L^2(\Omega)$ and any $u\in U=U_M:=\set{u\in L^2(\Omega)}{u_1\leq u\leq M \text{ a.e.}}$, the solution to \eqref{eq:pot_state} satisfies $y\in H^2(\Omega)$ together with the uniform a priori estimate
\begin{equation}
    \label{eq:pot_apriori}
    \norm{y}_{H^2(\Omega)} \leq C_M \norm{f}_{L^2(\Omega)}.
\end{equation}
We also consider for given $u\in U_M$ and $y\in H^2(\Omega)$ the adjoint equation
\begin{equation}\label{eq:pot_adjoint}
    \left\{\begin{aligned}
            -\Delta w + u w &= -(y - z) &\text{ in } \Omega,\\
            \partial_\nu w &= 0&\text{ on } \partial \Omega,
    \end{aligned}\right.
\end{equation}
whose solution $w\in H^2(\Omega)$ also satisfies the uniform a priori estimate \eqref{eq:pot_apriori}.
Due to the Sobolev embedding theorem, we have that the solutions $y$ and $w$ are also bounded in $L^\infty(\Omega)$ uniformly with respect to $u\in U_M$.

By standard Lagrangian calculus, we can now write $p_\gamma = y_\gamma w_\gamma$, where $w_\gamma\in H^1(\Omega)$ is the solution to \eqref{eq:pot_adjoint} with $u=u_\gamma$ and $y=y_\gamma$.
We further eliminate $u_\gamma$ using the second equation of \eqref{eq:opt_reg} to obtain the reduced system
\begin{equation}
    \label{eq:opt_reg_potential}
    \left\{\begin{aligned}
            -\Delta w_\gamma + H_\gamma(-y_\gamma w_\gamma) w_\gamma + y_\gamma &= z,\\
            -\Delta y_\gamma +  H_\gamma(-y_\gamma w_\gamma) y_\gamma &= f.
    \end{aligned}\right.
\end{equation}
Due the regularity of $y_\gamma$ and $p_\gamma$, we can consider this as an equation in $L^2(\Omega)\times L^2(\Omega)$ for $(y_\gamma,p_\gamma)\in H^2(\Omega)\times H^2(\Omega)$. By the Sobolev embedding theorem, we have $y_\gamma w_\gamma\in L^\infty(\Omega)$, and hence that the system \eqref{eq:opt_reg_potential} is semismooth. By the chain rule, the Newton derivative of $H_\gamma(-yw)$ with respect to $y$ in direction $\delta y$ is given by
\begin{equation}
    D_{N,y} H_\gamma (-yw)\delta y = -\frac1\gamma \chi(-yw)\, w\, \delta y,
\end{equation}
where $\chi(-yw)$ is the characteristic function of the inactive set
\begin{equation}
    \calS_\gamma(-yw):=\bigcup_{i=1}^{d-1} \set{x\in\Omega}{-y(x)w(x)\in Q_{i,i+1}^\gamma}.
\end{equation}
Similarly,
\begin{equation}
    D_{N,w} H_\gamma (-yw)\delta w = -\frac1\gamma \chi(-yw)\, y\, \delta w.
\end{equation}
For convenience, we set $\chi^k := \chi(-y^kw^k)$.
A Newton step consists in solving
\begin{equation}
    \label{eq:newton_full_pot}
    \begin{multlined}[t]
        \begin{pmatrix}
            1-\tfrac1\gamma \chi^k (w^k)^2& -\Delta + H_\gamma(-y^kw^k)-\tfrac1\gamma \chi^ky^kw^k\\[1.5ex]
            -\Delta + H_\gamma(-y^kw^k) -\tfrac1\gamma \chi^ky^kw^k & -\tfrac1\gamma \chi^k (y^k)^2
        \end{pmatrix}
        \begin{pmatrix}
            \delta y \\[1.5ex] \delta w
        \end{pmatrix}
        \\= -
        \begin{pmatrix}
            -\Delta w^k + H_\gamma(-y^kw^k) w^k + y^k- z\\[1.5ex]
            -\Delta y^k + H_\gamma(-y^kw^k) y^k - f
        \end{pmatrix}
    \end{multlined}
\end{equation}
and setting $y^{k+1}=y^k+\delta y$ and $w^{k+1}=w^k+\delta w$.

To show local superlinear convergence, it remains to prove uniformly bounded invertibility of \eqref{eq:newton_full_pot}.
We proceed in several steps. First, we consider the off-diagonal terms in \eqref{eq:newton_full_pot}.
\begin{lemma}\label{lem:B}
    For any $\gamma>0$ and $y,w\in H^2(\Omega)$, the linear operator $B:H^2(\Omega)\to L^2(\Omega)$,
    \begin{equation}
        B =  -\Delta + H_\gamma(-yw)-\tfrac1\gamma \chi(-yw) yw,
    \end{equation}
    is uniformly invertible, and there exists a constant $C>0$ independent of $y,w$ such that
    \begin{equation}
        \norm{B^{-1}}_{\calL(L^2(\Omega),H^2(\Omega))} \leq C.
    \end{equation}
\end{lemma}
\begin{proof}
    We first note that by definition, $[H_\gamma(p)](x) \in [u_1,u_d]$ for any $p\in L^2(\Omega)$. Furthermore, on the inactive set $S_\gamma(-yw)$ we have, again by definition,
    \begin{equation}
        u_1 \leq \frac\alpha{2\gamma}(u_1+u_2)+u_1 \leq \frac1\gamma(-yw)(x) \leq  \frac\alpha{2\gamma}(u_{d-1}+u_d)+u_d \leq (1+\tfrac\alpha\gamma)u_d.
    \end{equation}
    Thus, $H_\gamma(-yw)-\tfrac1\gamma \chi(-yw) yw\in U_M$ for $M=(2+\tfrac\alpha\gamma)u_d$, and the claim follows from the a priori estimate \eqref{eq:pot_apriori}.
\end{proof}

\begin{proposition}\label{lem:newton_bound}
    For $\gamma>0$, let $(y_\gamma,w_\gamma)\in H^2(\Omega)\times H^2(\Omega)$ be a solution to \eqref{eq:opt_reg_potential} with $w_\gamma$ satisfying $\|w_\gamma\|_{L^\infty(\Omega)}< \sqrt{\gamma}$. Furthermore, let $U(y_\gamma)$ be a bounded neighborhood of $y_\gamma$ in $H^2(\Omega)$, and let $U(w_\gamma)$ be a bounded neighborhood of $w_\gamma$ in $H^2(\Omega)$ such that $\|w\|_{L^\infty(\Omega)}\leq \sqrt{\gamma}$ for any $w\in U(w_\gamma)$.
    Then there exists a constant $C>0$ such that for any $(y,w)\in U(y_\gamma)\times U(w_\gamma)$ and any $r_1,r_2\in L^2(\Omega)$, there exists a unique solution $(\delta y,\delta w)\in H^2(\Omega)\times H^2(\Omega)$ to
    \begin{equation}\label{eq:newton_step_unif}
        \begin{pmatrix}
            1-\tfrac1\gamma \chi(-yw) w^2& B\\[1ex]
            B & -\tfrac1\gamma \chi(-yw) y^2
        \end{pmatrix}
        \begin{pmatrix}
            \delta y \\[1ex] \delta w
        \end{pmatrix}
        =
        \begin{pmatrix}
            r_1\\[1ex]
            r_2
        \end{pmatrix}
    \end{equation}
    satisfying
    \begin{equation}
        \norm{\delta y}_{H^2(\Omega)}+\norm{\delta w}_{H^2(\Omega)}\leq C \left(\norm{r_1}_{L^2(\Omega)}+\norm{r_2}_{L^2(\Omega)}\right).
    \end{equation}
\end{proposition}
\begin{proof}
    We exploit the invertibility of $B$ to obtain the required bounds on $\delta y$ and $\delta w$. For the sake of convenience, we set  $\omega := \calS_\gamma(-yw)$ and $h:=1-\frac1\gamma\chi(-yw)w^2$.
    As a first step, we introduce the following bilinear form on $L^2(\omega)\times L^2(\omega)$:
    \begin{equation}
        a_\omega(w_1,w_2) := \left( w_1,w_2\right)_{L^2(\omega)} + \left(h B^{-1}(\tfrac1{\sqrt{\gamma}}yE_{\omega} w_1), B^{-1}(\tfrac1{\sqrt{\gamma}}yE_{\omega} w_2)\right)_{L^2(\Omega)},
    \end{equation}
    where $E_{\omega}$ denotes the extension by zero operator from $\omega$ to $\Omega$.
    Due to the assumption on $w$, we have that $h$ ia nonnegative. Thus the second term on the right hand side of the above equation is non-negative as well. Hence $a_\omega$ is symmetric, continuous and elliptic on $L^2(\omega)$ (uniformly on the set of admissible $(y,w)$). This implies the existence of a unique solution $\delta\tilde w\in L^2(\omega)$ to
    \begin{equation}
        \label{eq:newton_step_aux}
        a_\omega(\delta\tilde w, \tilde w) = \left(\tfrac{1}{\sqrt{\gamma}}y B^{-1}\left(r_1 -h B^{-1}r_2\right),\tilde w\right)_{L^2(\omega)}
        \quad\text{for all }\tilde w\in L^2(\omega)
    \end{equation}
    satisfying
    \begin{equation}
        \norm{\delta\tilde w}_{L^2(\omega)}\leq C \left(\norm{r_1}_{L^2(\Omega)}+ \norm{r_2}_{L^2(\Omega)}\right).
    \end{equation}
    (Here and below, $C$ is a generic constant that may change its value between occurences but does not depend on $y$ and $w$.)

    Next we consider the auxiliary equation
    \begin{equation}
        \label{eq:newton_step_aux2}
        B\delta y 
        = r_2 + \tfrac{1}{\sqrt{\gamma}} y E_{\omega}\delta \tilde w.
    \end{equation}
    From \cref{lem:B} we obtain a unique solution $\delta y \in H^2(\Omega)$ to \eqref{eq:newton_step_aux2} satisfying
    \begin{equation}
        \norm{\delta y}_{H^2(\Omega)} \leq C \left(\norm{r_2}_{L^2(\Omega)} + \tfrac{1}{\sqrt{\gamma}}\norm{\delta\tilde w}_{L^2(\omega)}\right) \leq C
        \left(\norm{r_1}_{L^2(\Omega)}+ \norm{r_2}_{L^2(\Omega)}\right),
    \end{equation}
    using that $y\in U(y_\gamma)$ is uniformly bounded in $L^\infty(\Omega)$. Given $\delta y\in H^2(\Omega)$, the first equation of \eqref{eq:newton_step_unif} now admits a unique solution $\delta w\in H^2(\Omega)$ satisfying
    \begin{equation}
        \norm{\delta w}_{H^2(\Omega)} \leq C  \left(\norm{r_1}_{L^2(\Omega)} + \norm{\delta y}_{L^2(\Omega)}\right) \leq C
        \left(\norm{r_1}_{L^2(\Omega)}+ \norm{r_2}_{L^2(\Omega)}\right),
    \end{equation}
    using the uniform boundedness of $w\in U(w_\gamma)$ in $L^\infty(\Omega)$.

    To complete the proof, it remains to verify that $\delta w = \frac{1}{\sqrt{\gamma}} y \delta\tilde w$ on $\omega$. For this purpose we note that by the first of equation of \eqref{eq:newton_step_unif} and \eqref{eq:newton_step_aux2},
    \begin{equation}
        \delta w + B^{-1}\left(h B^{-1} \left(\frac{1}{\sqrt{\gamma}}yE_{\omega}\delta\tilde w\right)\right) = B^{-1}\left(r_1 -h B^{-1}r_2\right).
    \end{equation}
    Taking the inner product of this equation in $L^2(\omega)$ with $\frac1\gamma y E_{\omega} w_2$ for arbitrary $w_2\in L^2(\omega)$ and subtracting \eqref{eq:newton_step_aux}, we arrive at
    \begin{equation}
        \left(\tfrac1\gamma y \delta w - \delta\tilde w,w_2\right)_{L^2(\omega)} = 0\qquad\text{for all }w_2\in L^2(\omega).
    \end{equation}
    Inserting into \eqref{eq:newton_step_aux2} now verifies the second equation of \eqref{eq:newton_step_unif}.
\end{proof}
We remark that according to the a priori estimate \eqref{eq:pot_apriori}, the required smallness of $w_\gamma$ corresponds to smallness of the tracking error $\norm{y_\gamma-z}_{L^2(\Omega)}$.
In the following we give an alternative sufficient condition for the uniform continuous invertibility of the Newton iteration matrix \eqref{eq:newton_step_unif} that does not rely on the smallness of $w_\gamma$. For this purpose, we set $\omega_\gamma := \calS_\gamma(-y_\gamma w_\gamma)$ and define
\begin{equation}
    \partial\omega_\gamma := \bigcup_{i=1}^{d-1} \set{x\in\Omega}{-y_\gamma(x) w_\gamma(x)\in \partial Q_{i,i+1}^\gamma}.
\end{equation}
We also introduce the compact self-adjoint operator
\begin{equation}
    C:L^2(\omega_\gamma)\to L^2(\omega_\gamma), \qquad
    C=\left(B^{-1}(\tfrac1{\sqrt{\gamma}}yE_{\omega_\gamma})\right)^*
    (h_\gamma\Id) \left(B^{-1}(\tfrac1{\sqrt{\gamma}}yE_{\omega_\gamma})\right),
\end{equation}
where $h_\gamma = 1-\tfrac1\gamma \chi(-y_\gamma w_\gamma) w_\gamma^2$ and $B=B(y_\gamma,w_\gamma)$.
We require the following two assumptions. 
\begin{enumerate}[label= (\textsc{h}\arabic{enumi}), ref=\textsc{h}\arabic{enumi},align=left]
    \item $-1\notin\sigma(C)$,\label{ass:h1}
    \item $|\partial\omega_\gamma| = 0.$\label{ass:h2}
\end{enumerate}
\begin{proposition}\label{lem:newton_bound2}
    For $\gamma>0$, let $(y_\gamma,w_\gamma)\in H^2(\Omega)\times H^2(\Omega)$ be a solution to \eqref{eq:opt_reg_potential} satisfying \eqref{ass:h1} and \eqref{ass:h2}. Then there exists a neighborhood  $U(y_\gamma)\times U(w_\gamma)$ of $(y_\gamma,w_\gamma)$ in $H^2(\Omega)\times H^2(\Omega)$ such that the conclusion of \cref{lem:newton_bound} holds.
\end{proposition}
\begin{proof}
    By \eqref{ass:h1} and as a consequence of the proof of \cref{lem:newton_bound}, the system matrix in \eqref{eq:newton_step_unif} is continuously invertible in $(y_\gamma,w_\gamma)$.
    Since the set of continuously invertible operators between Hilbert spaces is open with respect to the topology of the operator norm
    (see, e.g., \cite[Theorem~6.2.3]{Wouk}),
    the claim will be established once we have argued that the system matrix, considered as an operator from $H^2(\Omega)\times H^2(\Omega)$ to $L^2(\Omega)\times L^2(\Omega)$, depends continuously in the operator norm on $(y,w)\in H^2(\Omega)\times H^2(\Omega)$ in a neighborhood of $(y_\gamma,w_\gamma)$.
    For this purpose, we first argue that $p:=-yw\mapsto \chi(p)$ is continuous from $C(\overline\Omega)$ to $L^2(\Omega)$ in a neighborhood of $p_\gamma := -y_\gamma w_\gamma$. For $\eps>0$ sufficiently small, we set
    \begin{equation}
        \partial\calS_\gamma^\eps := \bigcup_{i=1}^{d-1} \set{x\in\Omega}{\mathrm{dist}\left(p_\gamma(x), \partial Q_{i,i+1}^\gamma\right)<\eps}.
    \end{equation}
    The family $\{ \partial\calS_\gamma^\eps \}_{\eps>0}$ is monotone with respect to set inclusion and satisfies
    \begin{equation}
        \lim_{\eps\to 0} \left|\partial\calS_\gamma^\eps\right| = \left|\lim_{\eps\to 0} \partial\calS_\gamma^\eps\right| = |\partial\calS_\gamma| = 0.
    \end{equation}
    For any $\eps>0$ and any $p\in C(\overline\Omega)$ such that $\|p-p_\gamma\|_{C(\overline\Omega)}<\frac\eps2$, we thus have
    \begin{equation}
        \begin{aligned}
            \|\chi(p)-\chi(p_\gamma)\|^2_{L^2(\Omega)} &= \int_{\Omega\setminus\partial\calS_\gamma^\eps} |\chi(p)(x)-\chi(p_\gamma)(x)|^2\,dx +
            \int_{\partial\calS_\gamma^\eps} |\chi(p)(x)-\chi(p_\gamma)(x)|^2\,dx \\
            &= 0 +  \left|\partial\calS_\gamma^\eps\right| \to 0 \qquad \text{for }\eps\to 0,
        \end{aligned}
    \end{equation}
    since $\mathrm{dist}\left(p(x),\partial Q_{i,i+1}^\gamma\right)<\frac\eps2$ on $\Omega\setminus\partial\calS_\gamma^\eps$ due to the choice of $p$. Due to the continuous embedding $H^2(\Omega)\hookrightarrow C(\overline\Omega)$, there exists $\eta=\eta(\eps)$ such that $\|y-y_\gamma\|_{H^2(\Omega)}<\eta$ and $\|w-w_\gamma\|_{H^2(\Omega)}<\eta$ implies $\|yw-y_\gamma w_\gamma\|_{C(\overline\Omega)}<\frac\eps2$. Hence $y w \to \chi(-yw)$ is  continuous from $H^2(\Omega) \times H^2(\Omega)$ to $L^2(\Omega)$.

    In a similar manner, one argues continuity of $H_\gamma$ from $H^2(\Omega)\times H^2(\Omega)$ to $L^2(\Omega)$, since the pointwise case distinction in the definition \eqref{eq:hgamma} can equivalently be expressed via the sum of characteristic functions. It follows from these considerations that the system matrix in \eqref{eq:newton_step_unif} as an operator from $H^2(\Omega)\times H^2(\Omega)$ to $L^2(\Omega) \times L^2(\Omega)$ depends continuous on $(y,w)\in H^2(\Omega)\times H^2(\Omega)$.
\end{proof}

Semismoothness of \eqref{eq:opt_reg_potential} together with \cref{lem:newton_bound} or \cref{lem:newton_bound2} now implies local convergence of the Newton iteration; see, e.g., \cite[Theorem 8.6]{Kunisch:2008a}.
\begin{theorem}\label{theorem_ssn1}
    Under the assumptions of either \cref{lem:newton_bound} or \cref{lem:newton_bound2},
    if $(y^0,w^0)$ is sufficiently close in $H^2(\Omega)\times H^2(\Omega)$ to a solution $(y_\gamma,w_\gamma)$ to \eqref{eq:opt_reg_potential}, the semismooth Newton iteration \eqref{eq:newton_step_unif} converges superlinearly in $H^2(\Omega)\times H^2(\Omega)$ to $(y_\gamma,w_\gamma)$.
\end{theorem}

\subsubsection{Diffusion problem}

We now consider the optimization of the leading coefficient. Here we are immediately faced with the difficulty that the state equation is not closed with respect to weak convergence of $u$ in $L^2(\Omega)$ or even weak-$*$ convergence in $L^\infty(\Omega)$; in particular, we cannot expect \eqref{ass:a1} to hold. This is a classical difficulty concerning the identification of diffusion coefficients when only pointwise bounds are available.  In this respect we recall results from \cite{Murat:1977}
where, for given data $z$, and inhomogeneities $f$ and $g$, examples for non-existence of solutions to the problem
\begin{equation}
    \min_{0<u_1 \le u \le u_2} \int_\Omega |y(u)-z|^2\, dx    
    \qquad \text{s.t }\ - \nabla \cdot ( u \nabla y)=f,\quad y|_{\partial\Omega} = g,
\end{equation}
are given, as well as the notion of H- and G-convergence \cite{Murat1997}.
To address this difficulty and thus to  ensure \eqref{ass:a2}, we  propose to introduce a \emph{local} bounded smoothing operator $G:L^2(\Omega)\to L^2(\Omega)$ with the property that its restrictions satisfy $G \in \calL(L^s(\Omega), W^{1,s}(\Omega))$ and $G^*\in \calL(W^{1,s}(\Omega),W^{1,s}(\Omega))$ for $s\in (n,\infty)$ and $G(U_M)\subset U_M$. This choice of $s$ guarantees that
$W^{1,s}(\Omega)$ embeds compactly into $C(\overline \Omega)$ and that $W^{1,s}(\Omega)$ is a Banach algebra. For example, we can choose $G$ as local averaging, i.e.,
\begin{equation}
    \label{eq:rho}
    [Gu](x) = \frac{1}{|B_\rho|}\int_{B_\rho}u(x+\xi)\,d\xi,
\end{equation}
where $B_\rho$ is a ball with radius $\rho>0$ and center at the origin, and $u$ is extended by $u_1$ outside of $\Omega$.

The corresponding state equation is
\begin{equation}\label{eq:dif_state}
    \left\{\begin{aligned}
            -\nabla\cdot(Gu\, \nabla y) &= f  &\text{ in } \Omega,\\
            y &= 0 &\text{ on } \partial \Omega.
    \end{aligned}\right.
\end{equation}
We assume  that $\Omega\subset \R^N$, $N\leq 3$, is sufficiently regular such that for any $f\in L^s(\Omega)$ and any $u\in U=U_M$ defined as above, the solution to \eqref{eq:dif_state} satisfies $y\in W^{2,s}(\Omega) \cap H^1_0(\Omega)$ together with the uniform a priori estimate
\begin{equation}
    \label{eq:div_apriori}
    \norm{y}_{W^{2,s}(\Omega)} \leq C_M \norm{f}_{L^s(\Omega)}.
\end{equation}
This is the natural $W^{2,s}(\Omega)$ regularity estimate for strongly elliptic equations, see \cite[page 191]{Ladyzhenskaya:1968}. Here we use that the set $G(U_M)$ is bounded in  $W^{1,s}(\Omega)$ and hence that elements in $G(U_M)$ have a uniform modulus of continuity (which affects the constant $C_M$). Setting $S:u\mapsto y$ in \eqref{eq:dif_state} and $Y=L^2(\Omega)$, the assumptions \eqref{ass:a1} and \eqref{ass:a1} are satisfied.
Digressing for a moment, we recall that our solutions to \eqref{eq:problem} and \eqref{eq:problem_reg} still depend on $G$, and in particular in the case of \eqref{eq:rho}, they depend on $\rho$. Let us denote this dependence by $u_\rho$. Then as $\rho \to 0$, these solution converge weakly in $L^s(\Omega)$ and $G$-converge to a -- possibly different -- limit which both satisfies the constraints involved in $U$ and appears as diffusion coefficient in the state equation; see, e.g., \cite[Chapter 1.3]{Allaire:2002}.

We next turn for given $z\in L^s(\Omega)$ and any  $u\in U_M$ and $y \in W^{2,s}(\Omega)$  to the adjoint equation
\begin{equation}\label{eq:dif_adjoint}
    \left\{\begin{aligned}
            -\nabla\cdot(Gu\, \nabla w) &=  -(y - z)  &\text{ in } \Omega, \\
            w &= 0&\text{ on } \partial \Omega,
    \end{aligned}\right.
\end{equation}
whose solution $w\in W^{2,s}(\Omega)\cap H^1_0(\Omega)$ also satisfies the uniform a priori estimate \eqref{eq:div_apriori}. We note that the solutions $y$ and $w$ satisfy $\nabla y\cdot\nabla w\in W^{1,s}(\Omega)$.

Using the solution $y_\gamma$ to \eqref{eq:dif_state} for $u=u_\gamma$ and the solution $w_\gamma$ to \eqref{eq:dif_adjoint} for $u=u_\gamma$ and $y=y_\gamma$,
we can write $p_\gamma = -G^*(\nabla y_\gamma \cdot \nabla w_\gamma) \in  W^{1,s}(\Omega)$ and thus express \eqref{eq:opt_reg} equivalently as
\begin{equation}
    \left\{\begin{aligned}
            -\nabla\cdot(Gu_\gamma \nabla w_\gamma) + y_\gamma &= z,\\
            u_\gamma - H_\gamma(-G^*(\nabla y_\gamma \cdot \nabla w_\gamma)) &=0,\\
            -\nabla\cdot(Gu_\gamma \nabla y_\gamma) &= f.
    \end{aligned}\right.
\end{equation}
After eliminating $u_\gamma$ using the second equation, the reduced system has the form
\begin{equation}
    \label{eq:opt_reg_diffusion}
    \left\{\begin{aligned}
            -\nabla\cdot\left(\left(GH_\gamma(-G^*(\nabla y_\gamma \cdot \nabla w_\gamma))\right) \,\nabla w_\gamma\right) + y_\gamma &= z,\\[0.5ex]
            -\nabla\cdot\left(\left(GH_\gamma(-G^*(\nabla y_\gamma \cdot \nabla w_\gamma))\right) \,\nabla y_\gamma\right) &= f.
    \end{aligned}\right.
\end{equation}
We consider this again as an equation in $L^s(\Omega)\times L^s(\Omega)$ for $(y_\gamma,p_\gamma)\in (W^{2,s}(\Omega)\cap H^1_0(\Omega)) \times (W^{2,s}(\Omega)\cap H^1_0(\Omega))$, and interpret $H_\gamma$ as bounded linear operator from $W^{1,s}(\Omega)$ to $L^s(\Omega)$.   This renders system \eqref{eq:opt_reg_diffusion} semismooth. Appealing again to the chain rule for Newton derivatives and introducing $\chi = \chi(-G^*(\nabla y \cdot \nabla w))$, we obtain the Newton system
\begin{equation}
    \label{eq:newton_full_dif}
    \begin{multlined}[t]
        \begin{pmatrix}
            \Id + A^k(w^k,\cdot,w^k) & -\nabla\cdot\left(Gu^k\, \nabla\cdot\right)+A^k(y^k,\cdot,w^k) \\[1.5ex]
            -\nabla\cdot\left(Gu^k\,\nabla\cdot\right)+A^k(w^k,\cdot,y^k) & A^k(y^k,\cdot,y^k)
        \end{pmatrix}
        \begin{pmatrix}
            \delta y \\[1.5ex] \delta w
        \end{pmatrix}
        \\= -
        \begin{pmatrix}
            -\nabla\cdot\left(Gu^k \,\nabla w^k\right) + y^k- z\\[1.5ex]
            -\nabla\cdot\left(Gu^k \,\nabla y^k\right)  - f
        \end{pmatrix},
    \end{multlined}
\end{equation}
where we have set $u^k:=H_\gamma(-G^*(\nabla y^k \cdot \nabla w^k))$ and
\begin{equation}
    A^k(v_1,v_2,v_3) := \nabla\cdot\left(G\left(\tfrac1\gamma\chi^k G^*(\nabla v_1\cdot \nabla v_2)\right)\, \nabla v_3\right).
\end{equation}
Note that for all $y,w,\delta y,\delta w\in H^2(\Omega)$,
\begin{equation}
    \left(A^k(y,\delta y,w),\delta w\right)_{L^2(\Omega)} = \left(A^k(w,\delta w,y),\delta y\right)_{L^2(\Omega)}.
\end{equation}

It remains to provide sufficient conditions for the uniform bounded invertibility of the system matrix in \eqref{eq:newton_full_dif}. For this purpose we specify the critical set $\partial\omega_\gamma$ for the present case:
\begin{equation}
    \partial\omega_\gamma := \bigcup_{i=1}^{d-1} \set{x\in\Omega}{-G^*(\nabla y_\gamma(x)\cdot  \nabla w_\gamma(x))\in \partial Q_{i,i+1}^\gamma}.
\end{equation}

\begin{theorem}\label{theorem_ssn2}
    Let $(y_\gamma,w_\gamma)$ denote a solution to  \eqref{eq:opt_reg_diffusion}, assume that $| \partial\omega_\gamma|=0$, and that the system matrix \eqref{eq:newton_full_dif} evaluated at $(y_\gamma,w_\gamma)$ is continuous invertible as an operator from $(W^{2,s}\cap H^1_0(\Omega))^2$ to $(L^s(\Omega))^2$. Then, if $(y^0,w^0)$ is sufficiently close in $(W^{2,s}\cap H^1_0(\Omega))^2 $  to $(y_\gamma,w_\gamma)$, the semismooth Newton iteration \eqref{eq:newton_step_unif} converges superlinearly to $(y_\gamma,w_\gamma)$.
\end{theorem}
\begin{proof} It suffices to argue that the system matrix depends continuously on $(y,w)\in (W^{2,s}(\Omega)\cap H^1_0(\Omega))^2$ in a neighborhood of $(y_\gamma,w_\gamma)$ considered as operators in $\mathcal{L}((W^{2,s}(\Omega)\cap H^1_0(\Omega))^2,L^s(\Omega)^2)$. For this purpose we consider the operator
    \begin{equation}
        (W^{2,s}(\Omega)\cap H^1_0(\Omega))^2\ni(y,w)\mapsto A(w,\cdot,w)\in \mathcal{L}(W^{2,s}(\Omega)\cap H^1_0(\Omega),L^s(\Omega)),
    \end{equation}
    where $A$ still depends on $\chi = \chi(-G^*(\nabla y \cdot \nabla w))$.  First  we argue exactly as in the proof of \cref{lem:newton_bound2} that
    \begin{equation}
        (W^{2,s}(\Omega)\cap H^1_0(\Omega))^2\ni (y,w)\mapsto \chi=\chi(-G^*(\nabla y \cdot \nabla w))\in L^s(\Omega)
    \end{equation}
    is continuous. Next we observe that
    \begin{equation}
        W^{2,s}(\Omega)\cap H^1_0(\Omega)\ni w \mapsto G^*(\nabla w \cdot \nabla \cdot ) \in \mathcal{L}(W^{2,s}(\Omega)\cap H^1_0(\Omega),W^{1,s}(\Omega))
    \end{equation}
    is continuous, and consequently
    \begin{equation}
        (W^{2,s}(\Omega)\cap H^1_0(\Omega))^2\ni(y,w)\mapsto G(\tfrac{1}{\gamma}\chi G^*(\nabla w \cdot \nabla \cdot )) \in \mathcal{L}(W^{2,s}(\Omega)\cap H^1_0(\Omega),L^{s}(\Omega))
    \end{equation}
    is continuous as well. From here we can conclude that
    $(y,w)\mapsto A(w,\cdot,w)$ is continuous from $(W^{2,s}(\Omega)\cap H^1_0(\Omega))^2$ to $\mathcal{L}((W^{2,s}(\Omega)\cap H^1_0(\Omega)),L^s(\Omega))$. We argue similarly for $A(w,\cdot,y)$, $A(y,\cdot,w)$ and $A(y,\cdot,y)$, which establishes the claim.
\end{proof}

Returning to the assumption on the well-posedness of the system matrix at $(y_\gamma, w_\gamma)$, we now argue that this is indeed the case if $w_\gamma$ is sufficiently small in the $W^{2,s}(\Omega)$ norm, i.e., for small residual problems. For $w=0$, the system matrix in \eqref{eq:newton_full_dif} has the form
\begin{equation}
    \begin{pmatrix}
        \Id  & -\nabla\cdot\left(u_1\, \nabla\cdot\right) \\[1.5ex]
        -\nabla\cdot\left(u_1\,\nabla\cdot\right) & 0
    \end{pmatrix}
\end{equation}
since $u_\gamma = GH_\gamma(0) = Gu_1 = u_1$ because $G u = u$ for $u$ constant.
This operator is clearly continuously invertible. A perturbation argument as in the proof of \cref{theorem_ssn2} implies continuous invertibility also for $(y_\gamma, w_\gamma)$ if $\norm{w_\gamma}_{W^{2,s}(\Omega)}$ is sufficiently small.

\section{Numerical examples}
\label{sec:examples}

We illustrate the behavior of the proposed approach with numerical examples modeling a simple material design problem for the potential and the diffusion equation, in which a reference binary material distribution $u_r$ (i.e., using only two values: matrix or void, and material) has already been obtained. The goal is now to obtain a comparable behavior using additionally available materials of intermediate density (and hence presumably lower cost) by solving the multi-material optimization problem \eqref{eq:problem} with target $z=y_r$ (the solution to the state equation corresponding to the reference coefficient $u_r$) and an extended list $u_b$ of feasible material parameters containing the two original values. Here, the tracking term $\calF$ penalizes the deviation from the reference state, while the ``multi-bang'' term $\calG$ both promotes the desired discrete structure and favors materials with lower density; the trade-off between the two goals is controlled by the parameter $\alpha$. We point out that not strictly enforcing attainment of the target allows parameter distributions that are different from the original binary distribution (which is only recovered in the limit $\alpha\to 0$).
For each example, we report on the deviation from the reference state as well as on the achieved total material cost reduction (as measured by the difference of the $L^2$ norms of the reference and computed coefficients).

The multi-material optimization problem \eqref{eq:problem} is solved using the described regularized semismooth Newton method. To address the local convergence of Newton methods and to avoid having to choose the Moreau--Yosida regularization parameter $\gamma$ a priori, a continuation strategy is applied where the problem is solved starting with a large $\gamma^0=1$ and the initial guess $(y_0,p_0)=(0,0)$. The regularization parameter is then successively reduced via $\gamma^{k+1}=\gamma^k/2$, taking the previous solution as a starting point. The iteration is terminated if $\gamma=10^{-12}$ is reached or more than $300$ Newton iterations are performed. This is combined with a non-monotone backtracking line seach based on the residual of the optimality system \eqref{eq:opt_reg}, starting with a step length of $1$ and using a reduction factor of $1/2$, where a minimal step length of $10^{-6}$ is accepted even if it leads to a (small) increase in the residual norm. 
The partial differential equations are discretized using finite differences on a uniform grid of $128\times 128$ grid points. Our Matlab implementation of the described algorithm can be downloaded from \url{https://github.com/clason/multimaterialcontrol}.

\subsection{Potential problem}

\begin{figure}[t]
    \centering
    \subcaptionbox{reference coefficient $u_r$ \label{fig:pot:true}}%
    {\includegraphics[width=0.495\textwidth]{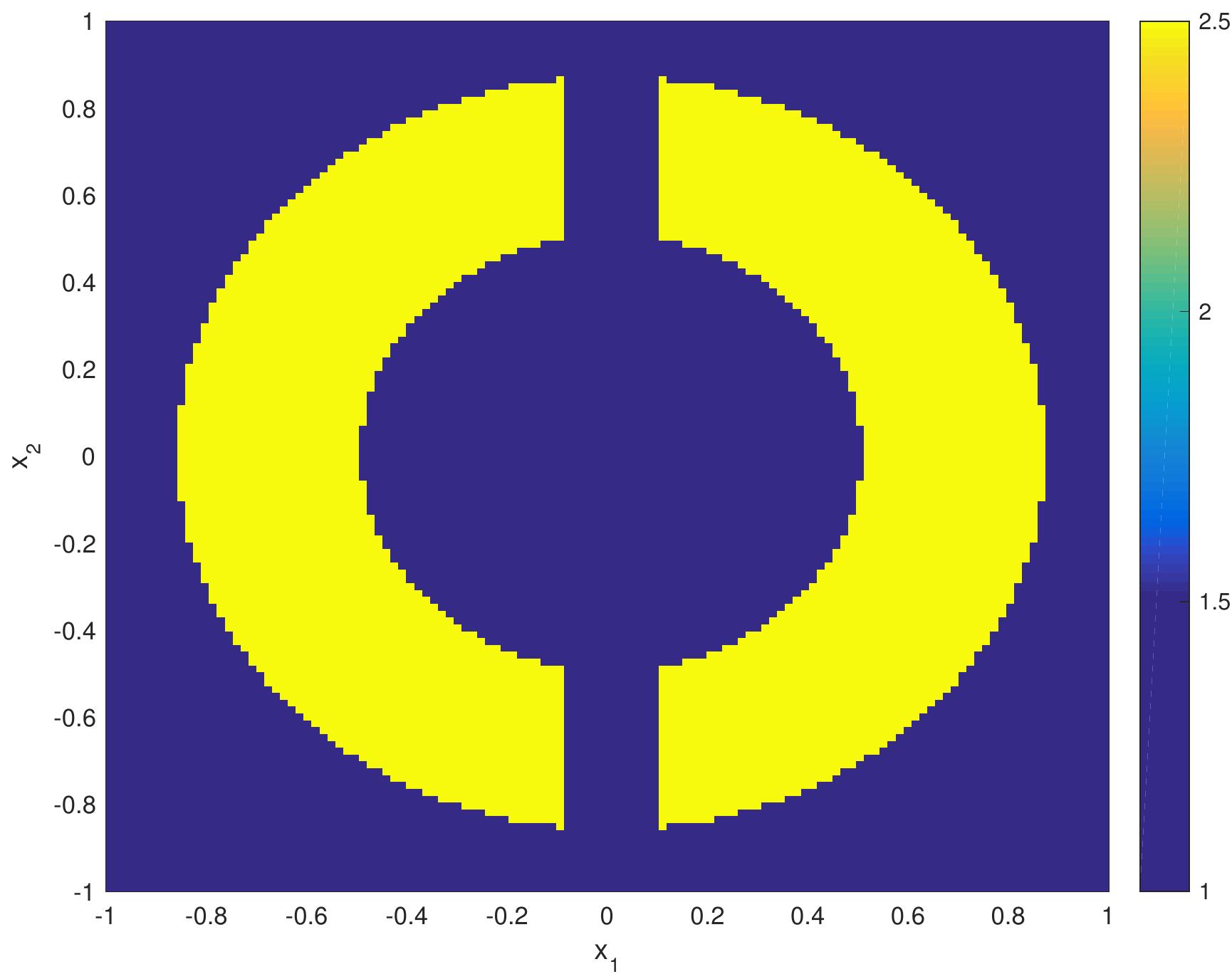}}
    \hfill
    \subcaptionbox{optimal coefficient $u_\gamma$ for $\alpha=10^{-5}$\label{fig:pot:control5}}%
    {\includegraphics[width=0.495\textwidth]{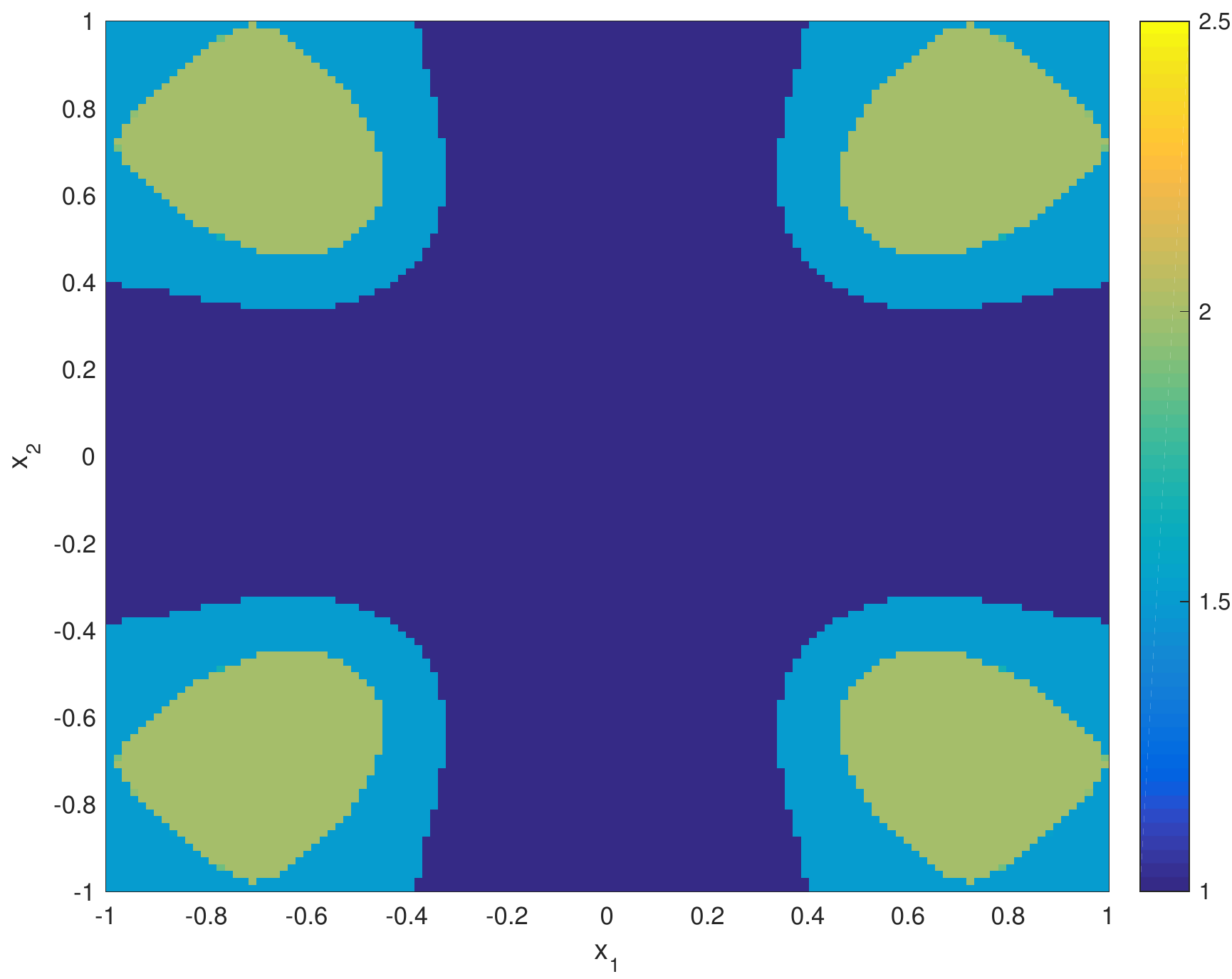}}

    \subcaptionbox{optimal coefficient $u_\gamma$ for $\alpha=10^{-6}$\label{fig:pot:control6}}%
    {\includegraphics[width=0.495\textwidth]{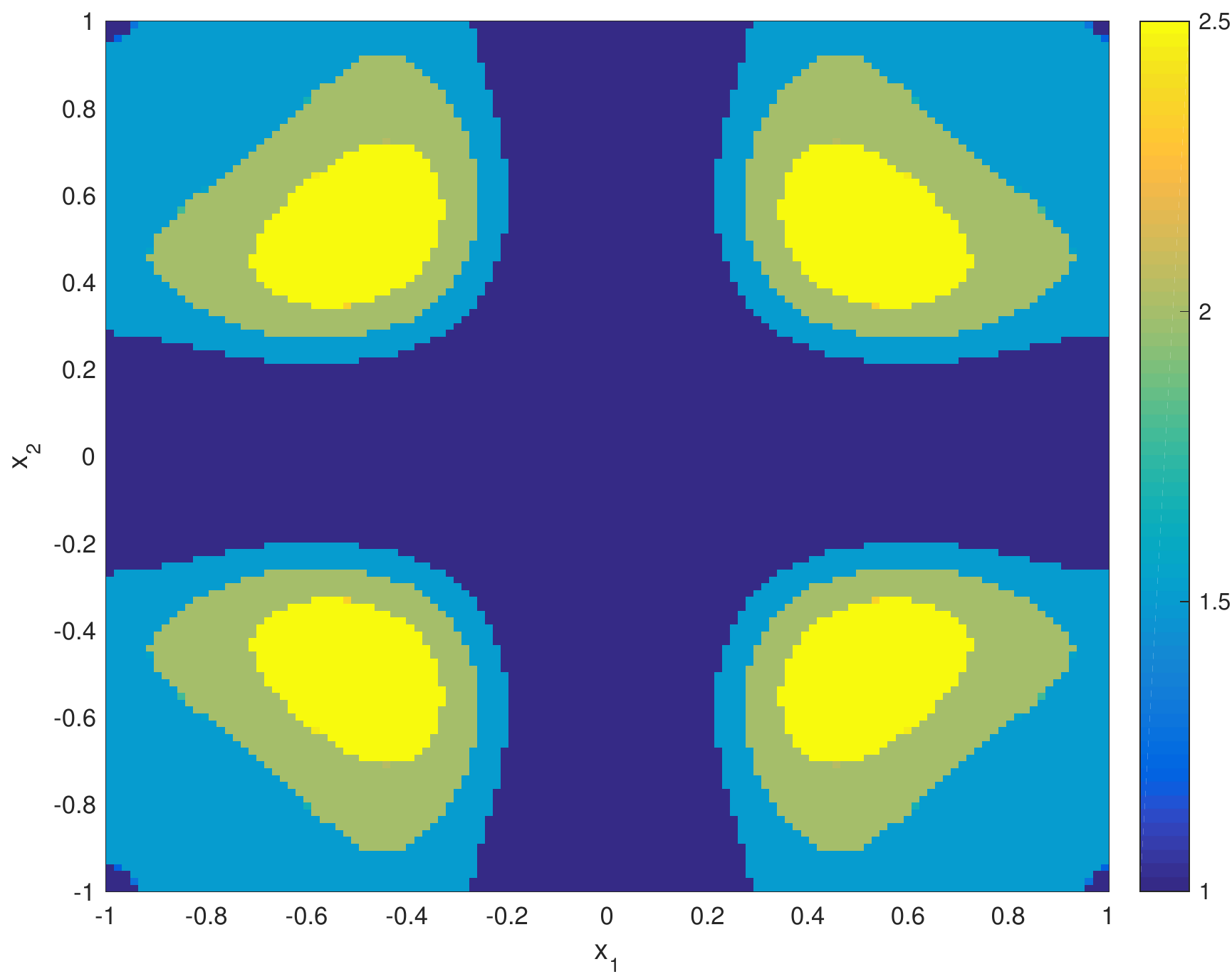}}
    \hfill
    \subcaptionbox{optimal coefficient $u_\gamma$ for $\alpha=10^{-7}$\label{fig:pot:control7}}%
    {\includegraphics[width=0.495\textwidth]{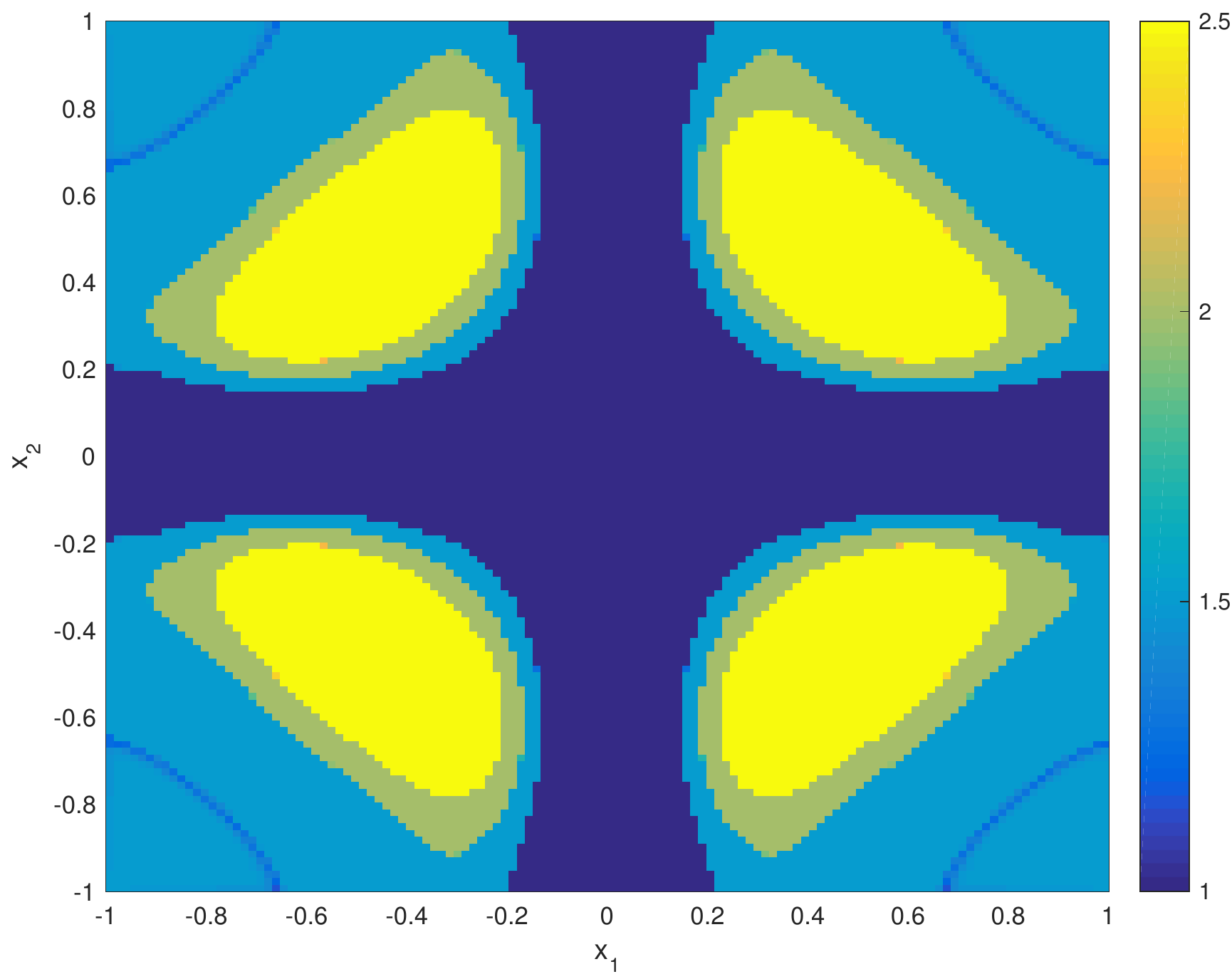}}
    \caption{Results for potential problem}
    \label{fig:pot}
\end{figure}
We first consider the design problem associated with equation \eqref{eq:pot_state}, where we fix $\Omega=[-1,1]^2$ and
\begin{equation}
    f(x_1,x_2) = \sin(\pi x_1)\cos(\pi x_2).
\end{equation}
The reference material parameter is
\begin{equation}
    \label{eq:reference}
    u_r (x_1,x_2) = \begin{cases}
        2.5 & \text{if } 1/4 < |x|^2 < \tfrac34 \text{ and } x_1>\tfrac{1}{10},\\
        2.5 & \text{if } 1/4 < |x|^2 < \tfrac34 \text{ and } x_1<-\tfrac{1}{10},\\
        1.5 & \text{else},
    \end{cases}
\end{equation}
see \cref{fig:pot:true}. We then solve the multi-material design problem for the target $z=y_r$ with the extended feasible parameter set $\{1,1.5,2,2.5\}$ for different values of $\alpha$ using the described algorithm.
In all cases, after some initial reduced steps were taken for $\gamma < 5\cdot 10^{-5}$, the Newton iteration entered a superlinear phase and converged after at most three iterations. Depending on $\gamma$, the total number of Newton iterations was between $5$ and $28$.
The algorithm always terminated at $\gamma\approx 10^{-12}$ because the minimal value of $\gamma$ was reached. 
The final material distributions $u_\gamma$ for $\alpha \in\{10^{-5},10^{-6},10^{-7}\}$ are shown in \cref{fig:pot:control5}--\subref{fig:pot:control7}. As can be seen, at almost all points, only the feasible parameter values are attained, where lower values of $\alpha$ lead to increased use of higher density materials. The relative tracking error $e_T:=\norm{y_\gamma-y_r}_{L^2}/\norm{y_r}_{L^2}$ as well as the relative total material cost reduction $e_M:=(\norm{u_r}_{L^2}-\norm{u_\gamma}_{L^2})/\norm{u_r}_{L^2}$ for each value of $\alpha$ are given in \cref{tab:pot}.

\subsection{Diffusion problem}

\begin{figure}[t]
    \centering
    \subcaptionbox{reference coefficient $Gu_r$ \label{fig:diff:true}}%
    {\includegraphics[width=0.495\textwidth]{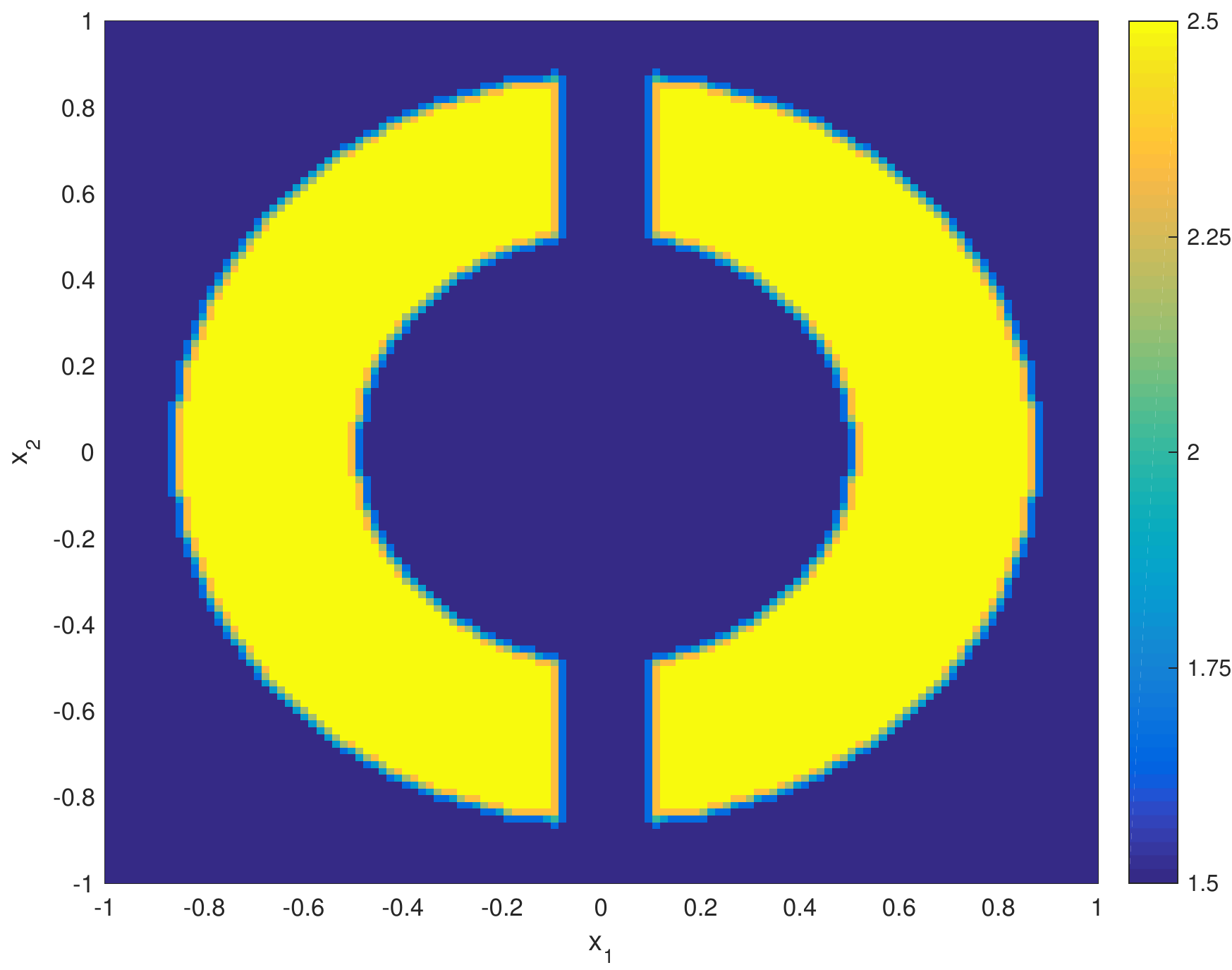}}
    \hfill
    \subcaptionbox{optimal coefficient $Gu_\gamma$ for $\alpha=10^{-2}$\label{fig:diff:control2}}%
    {\includegraphics[width=0.495\textwidth]{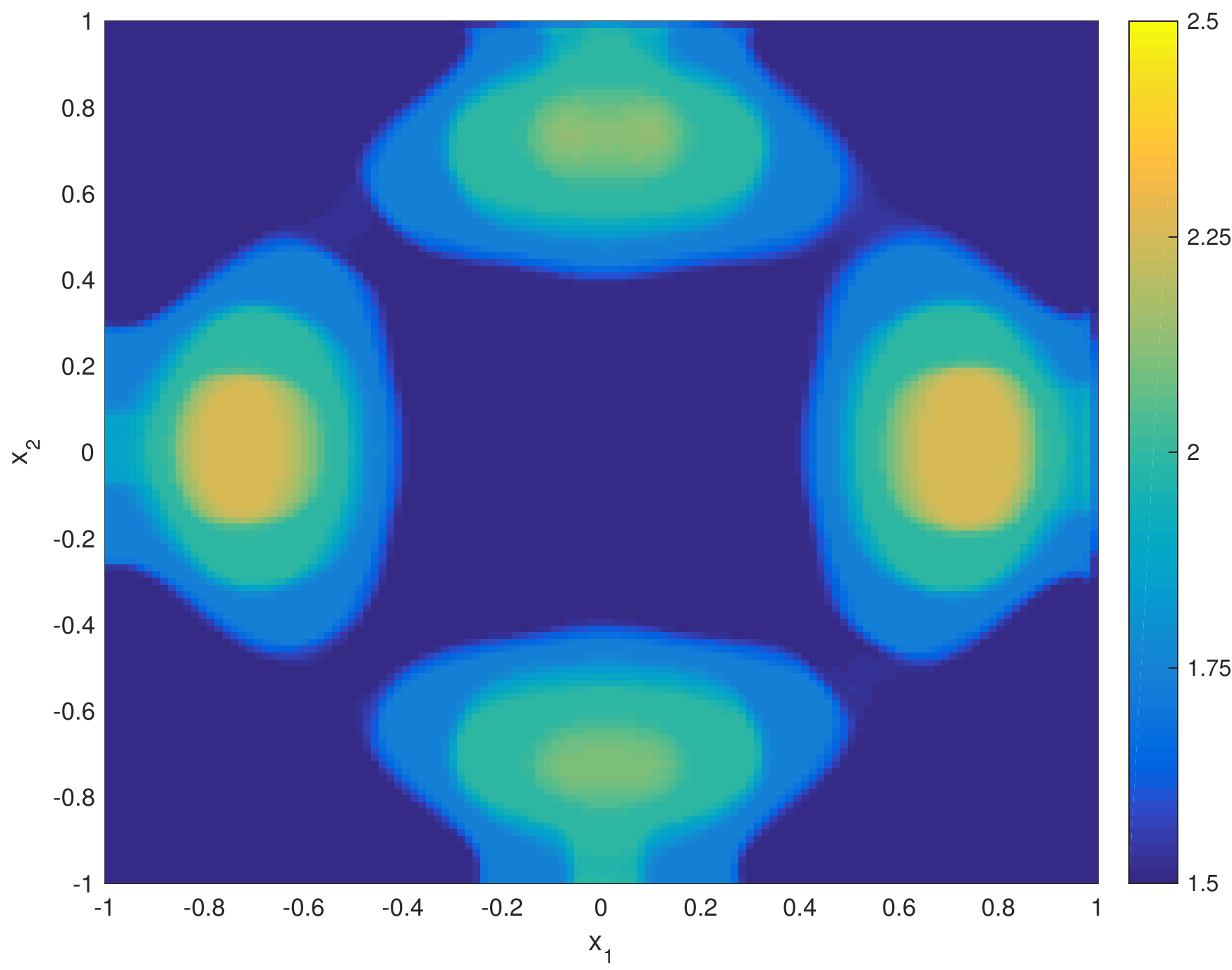}}

    \subcaptionbox{optimal coefficient $Gu_\gamma$ for $\alpha=10^{-3}$\label{fig:diff:control3}}%
    {\includegraphics[width=0.495\textwidth]{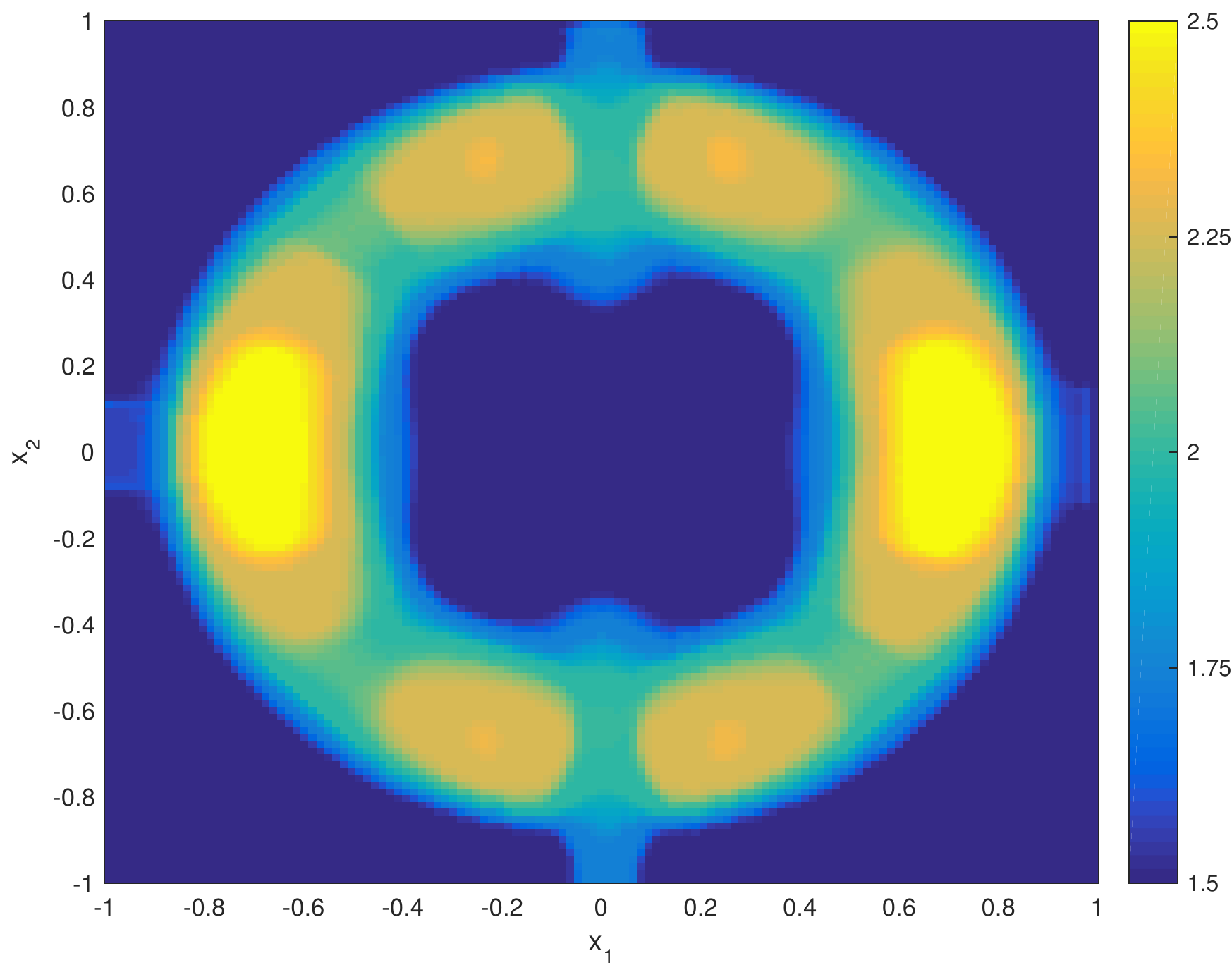}}
    \hfill
    \subcaptionbox{optimal coefficient $Gu_\gamma$ for $\alpha=10^{-6}$\label{fig:diff:control6}}%
    {\includegraphics[width=0.495\textwidth]{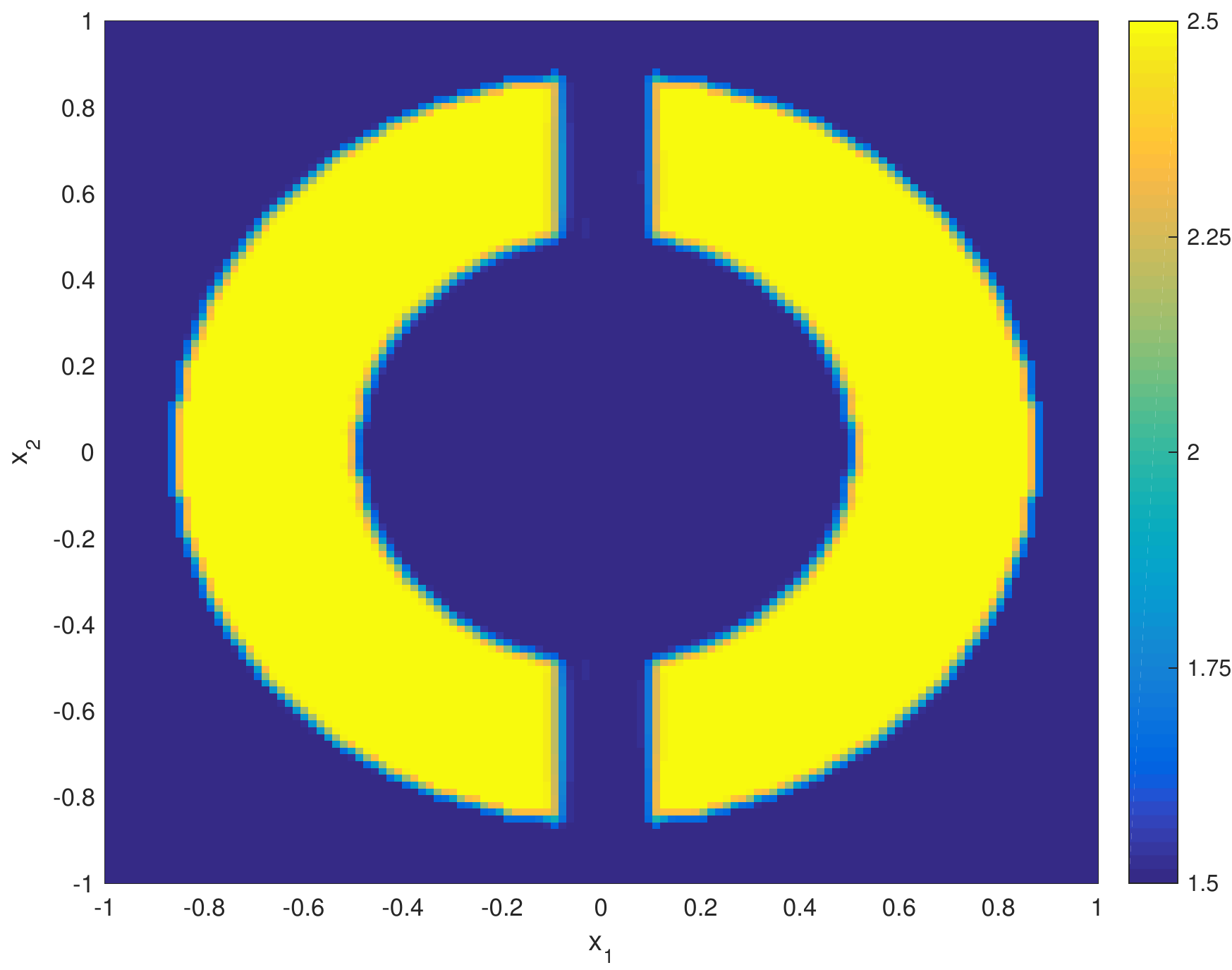}}
    \caption{Results for diffusion problem}
    \label{fig:diff}
\end{figure}
For the design problem associated with equation \eqref{eq:dif_state}, we set $f\equiv 10$ and $u_r$ as given in \eqref{eq:reference}. The smoothing operator $G$ is taken as averaging over the local five-point stencil; the smoothed reference coefficient $Gu_r$ is shown in \cref{fig:diff:true} to facilitate comparison.
For the multimaterial design problem, we choose the extended feasible parameter set $\{1.5,1.75,2,2.25,2.5\}$ and $\alpha \in\{10^{-2},10^{-3},10^{-6}\}$ (the last value to illustrate the behavior for $\alpha\to 0$).
In these cases, the algorithm terminated prematurely due to reaching the maximal number of Newton iterations at $\gamma^* \approx 4.8\cdot10^{-7}$, $\gamma^* \approx 6.0\cdot10^{-8}$, and $\gamma^* \approx 9.3\cdot10^{-10}$, respectively. The behavior of the Newton method is similar as in the potential problem, although the required number of Newton iterations now increases significantly as $\gamma$ is decreased due to the line search leading to smaller step lengths (including, e.g., for $\alpha=10^{-3}$ in total six non-monotone steps due to the minimal step length being reached).
The corresponding material coefficients $Gu_\gamma$ from the last successful iteration at $\gamma = 2\gamma^*$ are shown in \cref{fig:diff:control2}--\subref{fig:diff:control6}. Although the multi-bang structure is no longer perfect, it can be observed that the penalty is successful in promoting the desired parameter values even in the presence of the smoothing operator $G$. \Cref{fig:diff:control6} also indicates that the original binary reference distribution $u_r$ is recovered for $\alpha\to 0$.
Finally, the relative tracking errors and relative material cost reductions for these values of $\alpha$ are given in \cref{tab:dif}.
\begin{table}[t]
    \captionabove{Relative tracking error $e_T$ and material cost reduction $e_M$ for different values of $\alpha$}\label{tab:res}
    \begin{subtable}{0.5\textwidth}
        \caption{Potential problem\label{tab:pot}}            \begin{tabular}{lccc}
            \toprule
            $\alpha$ & $10^{-5}$ & $10^{-6}$ & $10^{-7}$ \\
            \midrule
            $e_T$ & $2.95\cdot10^{-2}$ & $8.28\cdot10^{-3}$ & $2.01\cdot10^{-3}$\\
            $e_M$ & $2.89\cdot10^{-1}$ & $1.82\cdot 10^{-1}$ & $1.10\cdot10^{-1}$\\
            \bottomrule
        \end{tabular}
    \end{subtable}\begin{subtable}{0.5\textwidth}
        \caption{Diffusion problem\label{tab:dif}}            \begin{tabular}{lccc}
            \toprule
            $\alpha$ & $10^{-1}$ & $10^{-2}$ & $10^{-6}$ \\
            \midrule
            $e_T$ & $4.96\cdot10^{-2}$ & $1.15\cdot10^{-2}$ & $5.29\cdot10^{-5}$\\
            $e_M$ & $1.16\cdot10^{-2}$ & $4.61\cdot 10^{-1}$ & $7.29\cdot10^{-4}$\\
            \bottomrule
        \end{tabular}
    \end{subtable}
\end{table}

\section{Conclusion}

A convex analysis approach is presented for the determination of piecewise constant coefficients in a partial differential equation where the constants range over a predetermined discrete set. Since the subdomains where the coefficient is constant are not specified a priori, this constitutes a topology optimization problem. 
Two model applications are analyzed in detail. For the case where the unknown coefficient enters into the potential term, the numerical results are very encouraging. If the unknown parameter enters into the diffusion term, regularization is required that has a smoothing effect on the solutions, and thus the numerical results are less ``crisp''. In practice, this could be addressed by a post-processing step, either by standard thresholding  or by evaluating the unregularized subdifferential at the computed optimal dual variable, i.e., taking an appropriate selection $\tilde u \in \partial\calG^*(p_\gamma)$.
Since the considered problems resemble inverse coefficient problems, it comes
as no surprise that the diffusion problem is more ill-posed than the potential problem. 

In future work, we plan to return to the diffusion problem and to formulate the multi-topology optimization problem based on a bounded variation framework using a functional including the total variation seminorm. It may also be of interest to search for other types of functionals which serve the purpose of multi-material topology optimization. 
In particular, we note that the currently used formulation in \eqref{eq:formal_prob} favors values $u(x)=u_i$ with small magnitude over other ones. Depending on the practical relevance of the $u_i$, this may not be a desired effect. In this case, functionals should be constructed that favor different criteria (e.g., the weight or the price of different materials) while still keeping the ``multi-bang'' property feature of promoting controls with values only from the given set.

\section*{Acknowledgment}
Support by Austrian Science Fund (FWF) under grant SFB {F}32 (SFB ``Mathematical Optimization and Applications in Biomedical Sciences'') is gratefully acknowledged.

\printbibliography

\end{document}